\documentclass[opre,nonblindrev]{informs3}
\OneAndAHalfSpacedXI

\usepackage[utf8]{inputenc}
\usepackage[english]{babel}
\usepackage[margin=1in]{geometry}
\usepackage{natbib}
\bibpunct[, ]{(}{)}{,}{a}{}{,}
\def\bibfont{\footnotesize}

\def\BIBand{and}
\usepackage{mathtools}
\usepackage{subcaption,placeins}
\usepackage{graphicx}
\usepackage{pgfplots}
\usepackage{algorithm}
\usepackage{algorithmic}
\usepackage{bm,bbm}
\usepackage{booktabs}
\usepackage{microtype}
\usepackage[normalem]{ulem}
\usepackage{multirow}
\usepackage{tikz}
\usepackage[hidelinks,hypertexnames=false]{hyperref}
\usepackage{etoolbox}

\newcommand{\zerodisplayskips}{%
  \setlength{\abovedisplayskip}{7pt}%
  \setlength{\belowdisplayskip}{5pt}%
  \setlength{\abovedisplayshortskip}{0pt}%
  \setlength{\belowdisplayshortskip}{0pt}}
\appto{\normalsize}{\zerodisplayskips}
\appto{\small}{\zerodisplayskips}
\appto{\footnotesize}{\zerodisplayskips}

\theoremstyle{TH}
\newtheorem{theorem}{Theorem}
\newtheorem{lemma}[theorem]{Lemma}
\newtheorem{proposition}[theorem]{Proposition}
\newtheorem{corollary}[theorem]{Corollary}

\theoremstyle{EX}

\newtheorem{definition}[theorem]{Definition}

\def\ECHowTheorems{}
\EquationsNumberedThrough

\let\ORproof\proof
\let\ORendproof\endproof
\renewenvironment{proof}[1][Proof.]{\ORproof{#1}}{\hfill\Halmos\ORendproof}

\newcommand{\algcomment}[1]{\hfill$\triangleright$~#1}

\let\argmin\relax

\DeclareMathOperator*{\argmin}{arg\,min}

\newcommand{\Test}{\mathrm{Test}}

\MANUSCRIPTNO{}
\begin{document}
\RUNAUTHOR{Author's names blinded for peer review}
\RUNTITLE{Stability-Regularized Cross-Validation}
\TITLE{Stability Regularized Cross-Validation}
\ARTICLEAUTHORS{%
\AUTHOR{Ryan Cory-Wright}
\AFF{Department of Analytics, Marketing and Operations, \\ Imperial Business School, Imperial College London\\London, UK\\
\EMAIL{r.cory-wright@imperial.ac.uk}}
\AUTHOR{Andr{\'e}s G{\'o}mez}
\AFF{Department of Industrial and Systems Engineering, \\
Viterbi School of Engineering, University of Southern California, \\CA, USA\\
\EMAIL{gomezand@usc.edu}}
}

\ABSTRACT{
We study hyperparameter selection via cross-validation, motivated by settings in which the selected $k$-fold cross-validation error underestimates the selected model's test error. We propose a nested $k$-fold cross-validation scheme that minimizes a weighted sum of the standard cross-validation error and empirical hypothesis stability, with the stability weight selected by an outer cross-validation loop. We derive a generalization bound that controls test error using the cross-validation error and empirical hypothesis stability, together with correction terms common to every candidate. In a stylized ridge regression setting, we then show that whether the stability weight improves expected test-set performance depends on the signal-to-noise ratio, with the greatest improvement when the signal-to-noise ratio is small. On $13$ real-world regression datasets, the procedure reduces the average out-of-sample MSE relative to $k$-fold cross-validation by $5.6\%$ for sparse regression and by $6.8\%$ for CART. A side benefit of an outer cross-validation loop is that the selected error estimates are better calibrated than those from ordinary $k$-fold cross-validation: the error estimate moves from 12.3\% below the selected test-set MSE to 5.8\% above it for sparse regression and from 10.3\% below to 0.2\% above it for CART.
}


\maketitle
\vspace{-10mm}

\section{Introduction}
Modern cross-validation, as formulated by \citet{stone1974cross, geisser1975predictive}, building on \cite{larson1931shrinkage, mosier1951problems}, performs two tasks: \emph{selecting} statistical models and \emph{predicting} their accuracy. Cross-validation excels at prediction. For a model fixed in advance, it accurately estimates the model's test error across a wide range of learners, loss functions, and data-generation processes \citep{arlot2010survey, homrighausen2013lasso, patil2021uniform}. Moreover, its estimates sharpen as the number of samples grows \citep{kearns1997algorithmic, bousquet2002stability}.

Model selection is more challenging. Cross-validation is simple, general, and usually reliable, making it a default selection procedure throughout statistics and machine learning. However, both theory \citep[e.g.,][]{shao1993linear, blum2015ladder, gupta2021small} and experiments \citep[e.g.,][]{reunanen2003overfitting, rao2008dangers, stephenson2021can, bates2021cross} document an \textit{adaptivity gap} between validation and test set performance, especially in settings with limited data. An adaptivity gap means a machine learning model's validation error is systematically lower than its test-set error. This underestimation is not universal \citep{roelofs2019meta, bates2021cross}, but it can be severe, as the following three examples demonstrate:  First, in a sonar classification example, \citet[Figure 1]{reunanen2003overfitting} shows that leave-one-out cross-validation produces a nearest-neighbor classifier with 3\% cross-validation error and 25\% test error. Second, on data with independent random labels, \citet[Table 1]{rao2008dangers} reports that selecting among 100,000 candidate classifiers gives an average leave-one-out error of 16.8\%, although the test-set error is 50\%. Third, \citet[Figure 7(a)]{cawley2010over} show that tuning a kernel ridge classifier on a 64-observation validation set yields a hold-out error of approximately 4\%, while the test-set error is around 20\%. 

One explanation for these adaptivity gaps is that validation sets yield approximately unbiased estimates of test-set performance for a \emph{fixed} set of hyperparameters. However, validation scores are random variables and subject to some variance. Therefore, the act of \emph{optimizing} the validation error selects hyperparameter combinations that disappoint out-of-sample. In the extreme, \citet{ng1997preventing} shows that minimizing the validation error over a sufficiently rich candidate family overfits the validation data, analogous to how an expressive learner overfits its training data in small-data settings. The tendency of selected models to underperform their validation scores occurs across statistics and operations research under several names, including ``post-decision surprise'' \citep{harrison1984decision}, the ``error maximization effect'' \citep{michaud1989markowitz}, the ``optimizer's curse'' \citep{smith2006optimizer}, and ``out-of-sample disappointment'' \citep{van2021data}. This raises the question: \emph{is it possible to improve cross-validation's performance by trading a small increase in validation error for a better value of a second criterion tied to generalization?}

In this work, we answer this question affirmatively. We propose a refinement of cross-validation that improves its performance by using the empirical hypothesis stability (formally defined in Section \ref{sec:notation}) as a second criterion. The empirical hypothesis stability measures how much a model's loss changes when we remove one fold from its training set. In our Theorem \ref{thm:genbound}, we prove that, with high probability, each candidate's test error is at most its $k$-fold cross-validation error, plus its empirical hypothesis stability, plus correction terms common to every candidate. Among models with comparable cross-validation error, the bound therefore favors the more stable one. This observation motivates our selection rule (Algorithm \ref{alg:nested_cv}): minimize a weighted sum of the cross-validation error and the empirical hypothesis stability, and set the weight by an outer, nested cross-validation loop. A weight of zero recovers ordinary cross-validation exactly, so our procedure departs from standard practice only when the data support a departure to improve out-of-sample performance.

\subsection{Literature Review}\label{ssec:litreview}
Our work connects to four streams of the operations research, statistics, and machine learning literature: (i) statistical learning theory that links stability to generalization, (ii) the use of stability or parsimony as a model-selection criterion, (iii) corrections for the optimism of cross-validation, and (iv) out-of-sample disappointment in data-driven optimization.

The idea that stability supports generalization is nearly as old as cross-validation itself. Stability arguments entered learning theory through the works of \citet{rogers1978finite} and \citet{devroye1979distribution}, who prove that leave-one-out estimates are accurate for learning rules that change little when a single training point is removed. \citet{kearns1997algorithmic} name this property hypothesis stability, introduce the weaker notion of error stability, and use it to prove sanity-check bounds for a broader class of algorithms. \citet{bousquet2002stability} give a loss-based formulation of hypothesis stability, introduce uniform stability, and derive quantitative generalization bounds. Our empirical hypothesis stability measure in Section \ref{sec:stability_adjustment} is inspired by, but differs from, their definition--in particular, ours can be computed from the training data alone, whereas theirs requires knowledge of the data-generating process. The ensuing literature shows that stable algorithms generalize \citep{mukherjee2006learning}, and a problem is learnable if and only if it admits a stable asymptotic empirical risk minimizer \citep{shalev2010learnability}. A parallel line of work studies the bias, variance, and uncertainty of cross-validation estimates \citep{bengio2004no, bayle2020cross, bates2021cross, austern2025asymptotics}. We refer to \citet{arlot2010survey} and \citet{lei2025modern} for surveys of cross-validation, and \citet{stephenson2021can} for an analysis of the cross-validation loss in ridge regression.

A second stream of literature develops model selection rules that supplement predictive accuracy with parsimony or stability. The use of stability in this context originates with \citet{breiman1996heuristics}, who argues that stabilized procedures yield better predictions in subset selection. \citet{yu2013stability} broadens this argument, advocating stability alongside predictive accuracy in model selection. Subsequent work incorporates stability into model selection in two ways. First, using stability to sharpen cross-validation estimates: \citet{kale2011cvstability} introduce the notion of mean-square stability and show that, for learning algorithms satisfying it, the variance of the $k$-fold cross-validation error can drop substantially relative to a single hold-out estimate. Relatedly, repeated-kCV aims to reduce the variability of cross-validation estimates by averaging over different random fold assignments \citep{kim2009estimating}. Second, building stability into the selection rule: \citet{limyu2016escv} propose estimation stability with cross-validation (ESCV), which restricts attention to models at least as regularized as the ordinary cross-validation choice and selects a model with low prediction instability according to their ES score. A common alternative is the one-standard-error rule \citep{breiman2017classification, hastie2009elements}, which selects the most parsimonious model with a cross-validation error within one standard error of the minimum cross-validation error. Closely related to the one-standard-error rule, \citet{lei2020cross} proposes cross-validation with confidence (CVC), which selects the most parsimonious set of hyperparameters that, at the $\alpha=0.05$ significance level, is not outperformed by any other candidate model in the confidence set. \citet{chen2021one} find that the one-standard-error rule often improves variable selection but frequently degrades predictive performance, suggesting that a fixed, heuristic degree of conservatism is not appropriate in all regimes. 
Relatedly, \citet{meinshausen2010stability} retain variables with high selection probabilities under subsampling, \citet{liu2010stability} choose the least regularized graphical model that remains sparse and replicable under resampling, and \citet{sun2013consistent} propose a general stability-based tuning criterion for penalized regression. These works are complementary to ours: like ESCV, our method combines CV with empirical stability for hyperparameter selection. Our approach adds a loss-based empirical hypothesis stability penalty to the CV objective and tunes its weight through nested CV.

A third stream proposes procedures for correcting the optimism of the cross-validation (CV) score. \citet{tibshirani2009bias} propose a bias correction for the minimum error in CV, and \citet{varma2006bias} document the bias of CV for performance estimation and advocate the use of nested CV, whose benefits and pitfalls \citet{cawley2010over, krstajic2014cross, bates2021cross} examine in detail. These approaches often significantly improve CV's estimate of the test-set error, and thus improve on CV at the \textit{prediction} task described at the start of this paper. However, they leave the selected hyperparameters unchanged and do not improve on CV at the \textit{selection} task. Our approach uses nesting to add a stability term to CV's objective, which changes the selected hyperparameters and lets our procedure improve at both tasks. 

Finally, the out-of-sample disappointment that motivates this paper is also a central concern in data-driven and distributionally robust optimization (DRO) \citep{van2021data,gupta2021small, gupta2024debiasing}, where optimizing an estimated objective plays the role that optimizing a validation score plays here. The DRO literature offers two broad responses. The first response is to debias the estimate: \citet{gupta2021small} invoke Stein's lemma to debias estimates of a policy's out-of-sample performance in small-data, large-scale linear optimization, and show that hold-out, $k$-fold, and leave-one-out cross-validation can all perform poorly in that regime. Similarly, \citet{gupta2024debiasing} debias a policy's in-sample performance by incorporating a variance-gradient correction term derived from sensitivity analysis. However, it is unclear how to extend their approach to our setting, as it applies to problems with linear objectives over subsets of $[0,1]^n$. The second response robustifies the problem rather than the estimate: robust and distributionally robust formulations of regression recover Lasso and ridge-type penalties, among others \citep{xu2010robust, bertsimas2018characterization, gao2024wasserstein}, and \citet{van2021data} prove that distributionally robust formulations optimally guard against out-of-sample disappointment for certain data-generation processes. Our stability penalty guards against out-of-sample disappointment from the opposite direction: rather than positing an ambiguity set over distributions, it measures a model's fragility directly from the data and lets nested cross-validation price that fragility.

\subsection{Contributions and Structure}
Our main contributions are fourfold. First, we prove a new generalization bound (Theorem \ref{thm:genbound}) on the expected test-set performance of a learner. Unlike the conceptually similar leave-one-out bound of \citet{bousquet2002stability}, which contains model-dependent terms that cannot be computed without knowledge of the data-generating process, our bound applies to $k$-fold cross-validation, and its model-dependent terms can be computed from the training data alone. Second, motivated by this bound, we propose \emph{stability-regularized cross-validation} in Algorithm \ref{alg:nested_cv}: a refinement of kCV that minimizes a weighted sum of the cross-validation error and the empirical hypothesis stability. A direct implementation requires on the order of $k\vert \Lambda\vert$ times as many model fits as ordinary $k$-fold cross-validation, where $\Lambda$ is the grid of possible stability parameters. Thus, we discuss computationally efficient implementations in Section \ref{ssec:nested_cv}. Third, we compare our procedure with $k$-fold cross-validation in a stylized ridge regression setting in Section \ref{sec:ridgeregression_theory} and prove that their expected test errors cross at a signal-to-noise threshold in the isotropic limit, with stability regularization strictly helping below this threshold and strictly hurting above it (Proposition \ref{prop:ref} and Theorem \ref{thm:phase}). Finally, we empirically evaluate our proposal on $13$ real-world datasets using sparse ridge regression and CART as the learners in Section \ref{sec:numres}, and find that, compared with kCV, our procedure reduces the average test-set MSE by 5.6\% for sparse ridge regression and by 6.8\% for CART.

The remainder of the paper is organized as follows. Section \ref{sec:stability_adjustment} introduces the empirical hypothesis stability, establishes our generalization bound (Theorem \ref{thm:genbound}), and develops the stability-regularized nested cross-validation procedure. Section \ref{sec:ridgeregression_theory} analyzes the performance of our scheme in a stylized ridge regression setting with a known data-generation process. Section \ref{sec:numres} evaluates the procedure on sparse ridge regression and CART across thirteen real-world datasets. Proofs and supplementary material are deferred to the appendices.

\section{Stability Regularized Cross-Validation}\label{sec:stability_adjustment}
{\color{black}
This section proposes a hyperparameter selection rule that improves the performance of $k$-fold cross-validation by penalizing model instability. Section \ref{sec:notation} defines our notation and introduces the notion of empirical hypothesis stability (Equation \eqref{eq:hypothesisStability2}), a sample-based measure of how much a model's loss changes when one fold is deleted from its training set. Section \ref{ssec:genbounds_kfold} then bounds each candidate's test error in terms of its $k$-fold error and empirical stability, plus a correction term common to any selected model (Theorem \ref{thm:genbound}). Section \ref{ssec:nested_cv} then converts our generalization bound into a selection rule which we formalize in Algorithm \ref{alg:nested_cv}: minimize a weighted sum of the cross-validation error and the empirical stability, with the weight set by a nested cross-validation scheme.

\subsection{Setup and Empirical Hypothesis Stability}\label{sec:notation}
We consider a supervised learning setting with $n$ data observations $\{(\bm{x}_i, y_i)\}_{i \in [n]} \in (\mathcal{X} \times \mathcal{Y})^n$ drawn i.i.d. from a fixed (but unknown) distribution $\mathcal{D}$, where $\mathcal{X} \subseteq \mathbb{R}^p$ denotes the feature space and $\mathcal{Y} \subseteq \mathbb{R}$ denotes the response space. We measure predictive accuracy through a loss function $\ell : \mathcal{Y} \times \mathcal{Y} \to [0, M]$, for instance, the squared loss $\ell(y, \hat{y}) = (y - \hat{y})^2$ when responses and predictions lie in a bounded set. We let $\{\mathcal{N}_j\}_{j \in [k]}$ denote a partition of the index set $[n]$ into $k$ folds, fixed independently of the data. In practice, the folds have approximately equal sizes and $k$ is typically $5$ or $10$, or $k = n$ for leave-one-out cross-validation, although our analysis requires neither equal fold sizes nor a particular $k$.

Let $\Theta$ denote a finite set of candidate hyperparameters. For each $\bm{\theta} \in \Theta$, we write $\bm{\beta}(\bm{\theta}) : \mathcal{X} \to \mathcal{Y}$ for the model trained on
the full dataset with hyperparameters $\theta$, and
$\bm{\beta}^{(\mathcal{N}_j)}(\bm{\theta}) : \mathcal{X} \to \mathcal{Y}$ for a model trained on all observations except those in fold
$\mathcal{N}_j$; for simplicity, we assume the learning algorithm returns a unique model for each training set and each $\bm{\theta}$. Overloading notation, we write
$\bm{\beta}(\bm{\theta}, \bm{x}_i) \in \mathcal{Y}$ for the prediction of
$\bm{\beta}(\bm{\theta})$ at $\bm{x}_i$.

For fixed hyperparameters $\bm{\theta} \in \Theta$, let the $k$-fold cross-validation error be defined as 
\begin{equation}\label{eq:cverror}
    h(\bm{\theta}) := \frac{1}{n} \sum_{j \in [k]}\sum_{i \in \mathcal{N}_j}
       \ell\left(y_i, \bm{\beta}^{(\mathcal{N}_j)}(\bm{\theta}, \bm{x}_i)\right).
\end{equation}

This equation leads to the following definition of how sensitive the fitted models at $\bm{\theta}$ are to deleting one fold of training data:
\begin{definition}\label{defn:empiricalhypothesisstability}
The \emph{empirical hypothesis
stability} of a learner, $\hat{\mu}_h$, is given by the following equation:
\begin{align}\label{eq:hypothesisStability2}\color{black}
	\hat{\mu}_h(\bm{\theta}):=\max_{j\in [k]}\frac{1}{n}\sum_{i=1}^n \left\vert \ell\left(y_i, \beta^{(\mathcal{N}_j)}(\bm{\theta}, \bm{x}_i)\right)-\ell\left(y_i, \beta(\bm{\theta}, \bm{x}_i)\right)\right\vert.
\end{align}
\end{definition}
In words, for each fold, Definition \ref{defn:empiricalhypothesisstability} records how much each observation's loss changes when that fold is deleted from the training set, averages these absolute changes over all $n$ observations, and reports the worst fold. A small value of $\hat{\mu}_h(\bm{\theta})$ indicates that no single fold exerts much influence on the fitted model, while a large value warns that the cross-validation score may rest on models that change substantially when a small fraction of the training data are removed.

Definition \ref{defn:empiricalhypothesisstability} is inspired by, but distinct from, the hypothesis stability of \citet[Definition~3]{bousquet2002stability}. Their definition yields generalization bounds on the test error in terms of the leave-one-out error, but because their stability parameter is an expectation under the unknown data-generating distribution, it cannot easily be evaluated in settings where we do not have knowledge of the underlying distribution. Equation~\eqref{eq:hypothesisStability2} modifies their definition in three ways: (i) it deletes an entire fold rather than a single observation, allowing their approach to be generalized from leave-one-out to $k$-fold cross-validation, (ii) it replaces a population expectation with an empirical average over the observed data, which is computable from the training data alone, and (iii) it takes a worst case over folds. In the next subsection, we show that these modifications yield a $k$-fold analog of \citet[]{bousquet2002stability}'s generalization bound, in which the model-dependent terms are computable from the data, at the price of correction terms common to every candidate.

\subsection{Generalization Bound}\label{ssec:genbounds_kfold}
We now bound each candidate model's expected test set error, with high probability, by its $k$-fold cross-validation error plus a correction term based on
its empirical hypothesis stability. 

To develop our bound, we require further setup and notation, which we now lay out. First, for any possible training dataset $\mathcal{T}:=(\bm{x}_i, y_i)_{i \in [n]}$, let $\beta(\bm{\theta})$ and
$\beta^{(\mathcal{N}_j)}(\bm{\theta})$ denote the models trained on $\mathcal{T}$ and on $\mathcal{T}$ with
fold $\mathcal{N}_j$ removed respectively, and let 
\begin{align*}
\Delta_{j,\bm{\theta}}(\bm{x},y):=
\left|
\ell\!\left(y,\beta^{(\mathcal{N}_j)}(\bm{\theta},\bm{x})\right)
-\ell\!\left(y,\beta(\bm{\theta},\bm{x})\right)
\right|
\end{align*}
denote the absolute change in the loss for the datapoint $(\bm{x}, y)$ after omitting the fold $\mathcal{N}_j$. We suppress the dependence of $\bm{\beta}$ and $\Delta$ on $\mathcal{T}$ for the rest of the section. Second, for any $a, q \in [0, 1]$ \begin{align}
    D_\mathrm{kl}(a \| q) := a \log\tfrac{a}{q} + (1 - a) \log\tfrac{1 - a}{1 - q}
\end{align} 
denotes the relative entropy between Bernoulli distributions with means $a$ and $q$, with the usual conventions at the boundary, that is, $0 \log 0:=0$ and $\log(u/0)=+\infty$ for $u>0$. Inverting the relative entropy in the second argument, the function
\begin{align}
q_+(a, c) := \sup\bigl\{ q \in [a, 1] : D_{\mathrm{kl}}(a \| q) \le c \bigr\}
\end{align}
then converts a tolerance $c\geq 0$ into an upper confidence bound for a mean. Third, let $\mathcal{S}:=((X_i,Y_i))_{i\in[n]}\sim\mathcal{D}^n$ denote the random training sample whose realization is the observed dataset $((\bm{x}_i,y_i))_{i\in[n]}$ introduced in Section~\ref{sec:notation}. Fourth, for $\delta\in(0,1)$, let $c_\Theta(\delta)$ be any deterministic number satisfying
\begin{align}\label{eq:ctheta}
\mathbb{P}_{\mathcal{S}}\Biggl\{\sup_{T, j, \theta} D_\mathrm{kl}\Biggl( \frac{1}{Mn} \sum_{i=1}^{n} \Delta_{j,\theta}(x_i, y_i) \Bigg\| \frac{1}{M} \mathbb{E}_{(x,y) \sim \mathcal{D}}\bigl[\Delta_{j,\theta}(x, y)\bigr] \Biggr) > c_\Theta(\delta) \Biggr\} \le \delta,
\end{align}
where the probability is taken over $\mathcal{S}$. The supremum over $\mathcal{T}$ ensures that, with probability at least $1 - \delta$, for every loss difference the learner can produce, the relative entropy between its normalized sample average and its normalized population mean is at most the radius $c_\Theta(\delta)$. 

We are now in a position to state our main result (proof deferred to Appendix \ref{append.genboundproof}):
\begin{theorem}\label{thm:genbound}
Suppose that the training data $(x_i, y_i)_{i \in [n]}$ are drawn i.i.d.\ from a distribution $\mathcal{D}$, that the fold partition and finite candidate set $\Theta$ are fixed independently of the data, and that $0 \le \ell(y, f(\bm{x})) \le M$ for every $(\bm{x}, y)$ and every model $f$ the learning algorithm can return. Then, the following bound on the test error holds with probability at least $1 - \delta$, simultaneously for each $\bm{\theta} \in \Theta$:
\begin{align}\label{eqn:genbound}
\mathbb{E}_{(x,y) \sim \mathcal{D}}\bigl[\ell\bigl(y, \bm{\beta}(\bm{\theta}, \bm{x})\bigr)\bigr]
\le
h(\bm{\theta})
+ M q_+\left( \frac{\hat{\mu}_h(\bm{\theta})}{M}, c_\Theta(\delta/2) \right)
+ M \sqrt{\frac{k \log(2k|\Theta|/\delta)}{2n}}.
\end{align}
\end{theorem}
We provide a family of instances certifying the tightness of Theorem \ref{thm:genbound} in Appendix \ref{ssec:worstcase}.

Theorem~\ref{thm:genbound} controls the amount by which $k$-fold cross-validation can underestimate the test set error. For fixed $\delta$ and $k$, if $\log \vert \Theta\vert /n\rightarrow 0$ and both $\hat{\mu}_h(\bm{\theta})$ and $c_{\Theta}(\delta/2)$ vanish as $n$ grows, then the upper correction terms vanish and the $k$-fold cross-validation error becomes an asymptotically valid upper bound on the test error. On the other hand, enlarging $\Theta$ can only reduce the attainable cross-validation error, yet it may loosen the overall bound through the $\log|\Theta|$ term and the radius. Finally, among candidates with equal cross-validation error, the most stable candidate enjoys the strongest bound, because $q_+(\cdot, c)$ is nondecreasing.

The stability correction in Theorem~\ref{thm:genbound}  is concave rather than linear. To connect it to the weighted selection rule used in the next subsection, fix $c := c_\Theta(\delta/2)$ and, for
$\lambda \ge 0$, let
\[
b_c(\lambda) := \sup_{a \in [0, 1]} \bigl\{ q_+(a, c) - \lambda a \bigr\}.
\] 
By definition, $b_c(\lambda)$ is the smallest intercept for which $q_{+}(a,c)\leq \lambda a+b_c(\lambda)$ holds at each stability level $a \in [0, 1]$. We then have the following result (proof deferred to Appendix \ref{ssec:proofcorrbound}):

\begin{corollary}\label{cor:affine}
In the setting of Theorem \ref{thm:genbound}, with probability at least
$1 - \delta$, the following holds simultaneously for every
$\bm{\theta} \in \Theta$ and every $\lambda \ge 0$:
\begin{equation}\label{eq:affine}
\mathbb{E}_{(x,y) \sim \mathcal{D}}\bigl[\ell\bigl(y, \beta(\bm{\theta}, x)\bigr)\bigr]
\le
h(\bm{\theta})
+ \lambda \hat{\mu}_h(\bm{\theta})
+ M b_c(\lambda)
+ M \sqrt{\frac{k \log(2k|\Theta|/\delta)}{2n}}.
\end{equation}
Moreover, $M b_c(\lambda)$ is the smallest intercept that makes the stability correction in \eqref{eq:affine} dominate that in Theorem \ref{thm:genbound} uniformly over all empirical stability levels in $[0, M]$, and the infimum of the right-hand side of \eqref{eq:affine} over $\lambda \ge 0$ recovers the right-hand side of \eqref{eqn:genbound}.
\end{corollary}

For a fixed $\lambda$, the final two terms in \eqref{eq:affine} do not depend on $\bm{\theta}$. Hence, minimizing $h(\bm{\theta})+\lambda\cdot\hat{\mu}_h(\bm{\theta})$ minimizes the corresponding upper bound. However, because the optimal choice of $\lambda$ depends on $\delta$, and the generalization bound may be conservative, Section~\ref{ssec:nested_cv} instead uses held-out data to compare the selection rules induced by different weights $\lambda$. 
}

\subsection{Stability Regularization and Nesting}\label{ssec:nested_cv}
For any fixed weight $\lambda \geq 0$, Theorem \ref{thm:genbound} and Corollary \ref{cor:affine} lead to the following hyperparameter selection rule 
\begin{align}\label{eq:param_opt_2}&\bm{\theta} \in \argmin_{\bm{\theta}\in \bm{\Theta}} \ h(\bm{\theta})+ \lambda \cdot \hat{\mu_h}(\bm{\theta}).
\end{align}

In particular, Equation \eqref{eq:param_opt_2} suggests selecting $\lambda$ from a finite set $\Lambda$ by an outer cross-validation loop. For each outer fold and each candidate weight, we solve \eqref{eq:param_opt_2} on the remaining observations and evaluate the selected model on the held-out fold. This separates the data used to select $\bm{\theta}$ from the data used to compare weights, as in standard nested cross-validation \citep{cawley2010over,bates2021cross}. We choose the weight with the lowest total outer loss. A weight of zero recovers ordinary $k$-fold cross-validation. Including $0 \in \Lambda$ therefore lets the procedure default to standard $k$-fold cross-validation whenever the data do not support a departure from it. Specifically, let $\mathcal{N}_{j,l}:=\mathcal{N}_j\cup\mathcal{N}_l$ denote the indices in the $j$th or $l$th fold. For each outer fold $j \in [k]$ and each $\lambda\in\Lambda$, let $\bm{\theta}_{j,\lambda}^\star$ denote a solution of
\begin{align}\label{eqn:nesting_lowerlevel}
\bm{\theta}_{j,\lambda}^\star\in\argmin_{\bm{\theta}\in\Theta}
 \Biggl\{
 \frac{1}{n-|\mathcal{N}_j|}
 \sum_{\substack{l\in[k]\\l\ne j}}
 \sum_{i\in\mathcal{N}_l}
 \ell\!\left(y_i,\bm{\beta}^{(\mathcal{N}_{j,l})}(\bm{\theta},\bm{x}_i)\right)
 +\lambda\hat\mu_h^{(-j)}(\bm{\theta})
 \Biggr\},
\end{align}
where
\begin{align}\label{eq:stab_inner}
\hat\mu_h^{(-j)}(\bm{\theta}) := \max_{\substack{l\in[k]\\l\ne j}} \frac{1}{n-|\mathcal{N}_j|}
 \sum_{i\in[n]\setminus\mathcal{N}_j}
 \left\vert \ell\!\left(y_i,\bm{\beta}^{(\mathcal{N}_{j,l})}(\bm{\theta},\bm{x}_i)\right)
 - \ell\!\left(y_i,\bm{\beta}^{(\mathcal{N}_j)}(\bm{\theta},\bm{x}_i)\right)\right\vert
\end{align}
is the empirical hypothesis stability \eqref{eq:hypothesisStability2} computed on the $n-|\mathcal{N}_j|$ observations outside fold $j$. We then select
\begin{align}\label{eqn:nesting_upperlevel}
 \lambda^\star\in\argmin_{\lambda\in\Lambda}
 \sum_{j\in[k]}\sum_{i\in\mathcal{N}_j}
 \ell\!\left(y_i,\bm{\beta}^{(\mathcal{N}_j)}(\bm{\theta}_{j,\lambda}^\star,\bm{x}_i)\right).
\end{align}
In instances where multiple weights attain the minimum outer score, we break ties by selecting the smallest $\lambda$. Finally, after fixing $\lambda^\star$, we apply \eqref{eq:param_opt_2} to the full sample. 

Algorithm~\ref{alg:nested_cv} summarizes the procedure. When Algorithm~\ref{alg:nested_cv} enumerates $\Theta$, the fitted models can be saved and reused across all $\lambda$. Alternatively, if $\Theta$ is explored adaptively via coordinate descent, as in our experiments, the evaluated candidates may depend on $\lambda$ and the nesting step is not redundant. We next give a generalization guarantee for Algorithm \ref{alg:nested_cv} and analyze its complexity.

\begin{algorithm}[h!]
\caption{Stability-regularized nested $k$-fold cross-validation for model selection}
\footnotesize
\label{alg:nested_cv}
\begin{algorithmic}[1]
\STATE \textbf{Input:} Dataset $\{(\mathbf{x}_i,y_i)\}_{i\in[n]}\in(\mathcal{X}\times\mathcal{Y})^n$; number of folds $k$; learning algorithm $\boldsymbol{\beta}$; loss function $\ell:\mathcal{Y}\times\mathcal{Y}\to\mathbb{R}_+$; hyperparameter grid $\Theta$; stability-weight grid $\Lambda$.
\STATE \textbf{Output:} Minimum outer score $\min_{\lambda\in\Lambda}\bar m_\lambda$, selected weight $\lambda^\star$, and selected hyperparameter $\bm{\theta}^\star$.
\STATE Randomly partition the dataset into $k$ disjoint folds indexed by $\mathcal{N}_1,\ldots,\mathcal{N}_k$.
\begingroup
\FOR{$\lambda\in\Lambda$}
  \FOR{$j\in[k]$}
    \STATE For each $\bm{\theta}\in\Theta$, fit $\bm{\beta}^{(\mathcal{N}_j)}(\bm{\theta})$ and $\bm{\beta}^{(\mathcal{N}_{j,l})}(\bm{\theta})$, $l\ne j$, by excluding the indicated folds.
    \STATE Compute $\hat\mu_h^{(-j)}(\bm{\theta})$ for $\bm{\theta}\in\Theta$ using \eqref{eq:stab_inner} and choose $\bm{\theta}_{j,\lambda}^\star$ using \eqref{eqn:nesting_lowerlevel}.
    \STATE $m_{\lambda,j}\gets\dfrac{1}{|\mathcal{N}_j|}\sum_{i\in\mathcal{N}_j}\ell\!\left(y_i,\bm{\beta}^{(\mathcal{N}_j)}(\bm{\theta}_{j,\lambda}^\star,\bm{x}_i)\right)$ \algcomment{outer-fold loss}
  \ENDFOR
  \STATE $\bar m_\lambda\gets\dfrac1n\sum_{j\in[k]}|\mathcal{N}_j|m_{\lambda,j}$ \algcomment{outer CV score}
\ENDFOR
\STATE Choose $\lambda^\star\in\argmin_{\lambda\in\Lambda}\bar m_\lambda$, as in \eqref{eqn:nesting_upperlevel}.
\STATE For each $\bm{\theta}\in\Theta$, fit $\bm{\beta}(\bm{\theta})$ on the full sample and compute $h(\bm{\theta})$ and $\hat\mu_h(\bm{\theta})$ via \eqref{eq:cverror}--\eqref{eq:hypothesisStability2}.
\STATE Choose $\bm{\theta}^\star$ using \eqref{eq:param_opt_2} with $\lambda=\lambda^\star$.
\STATE \textbf{return} $\min_{\lambda\in\Lambda}\bar m_\lambda,\lambda^\star,\bm{\theta}^\star$
\endgroup
\end{algorithmic}
\end{algorithm}

{\textbf{Generalization guarantee.}  Algorithm~\ref{alg:nested_cv} satisfies a guarantee analogous to Theorem \ref{thm:genbound}. Indeed, the outer loop is itself an instance of $k$-fold cross-validation in which each candidate is a \textit{selection policy} indexed by a meta-parameter $\lambda \in \Lambda$. Accordingly, each of Theorem \ref{thm:genbound}'s main ingredients has a direct counterpart. First, we replace the $k$-fold cross-validation error with the outer score $\bar{m}_{\lambda}:= \dfrac1n\sum_{t=1}^k|\mathcal{N}_t|m_{\lambda,t}$, which Algorithm \ref{alg:nested_cv} computes and \eqref{eqn:nesting_upperlevel} minimizes. Second, letting $\bm{\theta}^\star=\bm{\theta}^\star_{\lambda}$ denote an optimal solution to \eqref{eq:param_opt_2} at weight $\lambda$, we replace the empirical hypothesis stability $\hat{\mu}_h$ with the \textit{meta hypothesis stability}
\begin{equation}\label{eq:musel}
    \hat{\mu}_{\rm meta}(\lambda) := \max_{j \in [k]}\ \frac{1}{n}\sum_{i=1}^{n}
    \Big| \ell\big(y_i,\ \bm{\beta}^{(\mathcal{N}_j)}(\bm{\theta}^\star_{j,\lambda}, \bm{x}_i)\big)
    - \ell\big(y_i,\ \bm{\beta}(\bm{\theta}^\star_\lambda, \bm{x}_i)\big) \Big|,
\end{equation}
which measures the change in loss after deleting a fold, reselecting the hyperparameters, and refitting the model. Finally, we replace the radius $c_\Theta$ in \eqref{eqn:genbound} with a radius for the loss differences that Algorithm~\ref{alg:nested_cv}
can generate. Namely, let
\begin{equation*}
    \Delta^{\rm meta}_{j,\lambda}(\bm{x}, y) := \Big| \ell\big(y,\ \bm{\beta}^{(\mathcal{N}_j)}(\bm{\theta}^\star_{j,\lambda}, \bm{x})\big)
    - \ell\big(y, \bm{\beta}(\bm{\theta}^\star_\lambda, \bm{x})\big) \Big|,
\end{equation*}
and, analogously to Section \ref{ssec:genbounds_kfold}, $c_\Lambda(\delta)$ denotes any deterministic quantity such that
\begin{equation}\label{eq:cLambda}
    \mathbb{P}_S\left\{ \sup_{T, j, \lambda}\ D_{\rm kl}\!\left(
    \frac{1}{Mn}\sum_{i=1}^{n} \Delta^{\rm meta}_{j,\lambda}(\bm{x}_i, y_i)\ \middle\|\
    \frac{1}{M}\mathbb{E}_{(\bm{x},y)\sim\mathcal{D}} \Delta^{\rm meta}_{j,\lambda}(\bm{x}, y)
    \right) > c_\Lambda(\delta) \right\} \le \delta,
\end{equation}
where the supremum over training sets $T$ ensures, as in the definition of $c_\Theta(\delta)$, that the event covers every loss difference the procedure can generate. 

We then have the following corollary to Theorem~\ref{thm:genbound}, whose proof is deferred to Appendix \ref{ssec:proofofcorrnested}:

\begin{corollary}\label{cor:nested}
In the setting of Theorem~\ref{thm:genbound}, let
$\Lambda \subset [0,\infty)$ be a finite set fixed independently of the data. Then, with probability at least $1-\delta$, simultaneously for
every $\lambda \in \Lambda$ we have:
\begin{equation}\label{eq:nestedbound} \mathbb{E}_{(\bm{x},y)\sim\mathcal{D}} \ell\big(y,\ \bm{\beta}(\bm{\theta}^\star_\lambda, \bm{x})\big)
    \le \bar{m}_\lambda
    + M q_+\!\left( \frac{\hat{\mu}_{\rm meta}(\lambda)}{M},\ c_\Lambda(\delta/2) \right)
    + M \sqrt{\frac{k \log\big( 2k|\Lambda|/\delta \big)}{2n} },
\end{equation}
\end{corollary}
where $\bar{m}_\lambda$ denotes the nested kCV score with stability parameter $\lambda$ (see Algorithm \ref{alg:nested_cv}).

Corollary \ref{cor:nested} reveals that the price of tuning $\lambda$ is a union bound over $\Lambda$ rather than $\Theta$ (typically $|\Lambda| \ll |\Theta|$), together with the meta-stability term, which is small when the entire select-then-refit pipeline is insensitive to deleting a fold. The nested score thus certifies the performance of the deployed model when hyperparameter selection is itself stable, and degrades gracefully, through
$\hat{\mu}_{\rm meta}(\lambda^\star)$, when it is not.
}



\textbf{Complexity analysis.} Standard $k$-fold cross-validation evaluates each $\bm{\theta}\in\Theta$ using $k$ model fits, for a total of $O(k|\Theta|)$ fits, plus one fit at the selected model. A direct implementation of Algorithm~\ref{alg:nested_cv} uses $O(|\Lambda|k^2|\Theta|)$ fits to select $\lambda^\star$, plus $O(k|\Theta|)$ further fits to select $\bm{\theta}^\star$. Thus, its worst-case overhead relative to standard cross-validation is $O(|\Lambda|k)$. If the inner search exhaustively evaluates a fixed finite grid, the fitted models can be reused across $\lambda$, reducing this overhead by a factor of $|\Lambda|$.

\section{Theoretical Analysis of Algorithm \ref{alg:nested_cv} in a Ridge Regression Setting}\label{sec:ridgeregression_theory}
This section compares Algorithm \ref{alg:nested_cv} with $k$-fold cross-validation in a stylized ridge regression setting that admits an exact analysis of the test-set error, because the ridge estimator is a closed-form function of the data. In this setting, we show that the improvement depends on the noise in the data-generating process. 

We study a $p$-dimensional linear regression model with Gaussian data that obeys
\begin{align}\label{eqn:datagenridge}
y_i=\bm{x}_i^\top\bm{\beta}_\star+\varepsilon_i,
  \ i\in[n], \ \varepsilon_i\overset{\text{i.i.d.}}{\sim}\mathcal N(0,\sigma^2),
\end{align}
where the coefficient vector $\bm{\beta}_\star\in\mathbb R^p$, the i.i.d. covariates $\bm{x}_1,\dots,\bm{x}_n\in\mathbb R^p$, and the noise are mutually independent with
\begin{align}
\bm{x}_i \overset{\text{i.i.d.}}{\sim}\mathcal N(\bm{0},\bm{\Sigma}_x): \ \bm{0} \prec \bm{\Sigma}_x\prec 2\nu^2 \bm{I}_p,\label{eq:fs:moma}\\
\bm{\beta}_\star \sim\mathcal N(\bm{0},\bm{\Sigma}_\beta): \ \bm{0} \prec \bm{\Sigma}_\beta\prec 2 \tau^2\bm{I}_p,
  \label{eq:fs:momb}
\end{align}
for some $\nu^2, \tau^2, \sigma^2>0$, where $\prec$ denotes the Loewner (positive-definite) order,  and $\nu^2, \tau^2$ are scale parameters for the covariates and the coefficients. The sample is then split into $k\geq3$ folds of equal size $m=n/k$, such that $m$ is an integer, fixed in advance. 

For any $\gamma >0$ and nonempty training index set $I \subseteq [n]$, we then fix the ridge estimator 
\begin{equation}\label{eq:fs:learner}
  \bm{\hat\beta}_I(\gamma)
  =\arg\min_{\bm{\beta}\in\mathbb R^p}
  \frac{|I|}{2\gamma}\|\bm{\beta}\|_2^2+\frac12\sum_{i\in I}(y_i-\bm{x}_i^\top \bm{\beta})^2,
\end{equation}
where the factor $\vert I\vert$ keeps the penalty proportional to the number of training observations when folds are removed. Throughout, $\bm{X} \in \mathbb{R}^{p \times n}$ denotes the matrix with columns $\bm{x}_1, \ldots, \bm{x}_n \in \mathbb{R}^p$.

With this data generation process, we compare the performance of Algorithm \ref{alg:nested_cv} against that of kCV with two ridge levels, $\{\gamma_1,\gamma_2\}$, with $0<\gamma_1<\gamma_2<\infty$ and two stability weights $\Lambda=\{0,\lambda_1\}$, with $\lambda_1>0$. Further, let $\mathrm{Test}(\gamma)$ denote the expected squared prediction error of the full-sample fit $\hat{\bm\beta}_{[n]}(\gamma)$ on an independent draw from \eqref{eqn:datagenridge}, conditional on the training sample.

For a fixed $n, p, k, \lambda_1$ and normalized covariance matrices $\bm{\Sigma}_x/\nu^2, \bm{\Sigma}_\beta/\tau^2$, the expected errors of both procedures divided by $\nu^2 \tau^2$ depend on $(\sigma^2, \tau^2, \nu^2, \gamma_1, \gamma_2)$ only through the following dimensionless quantities (see Lemma \ref{lem:fs:norm} in Appendix \ref{ssec:omittedproofsec3.1}, for a dimensionality analysis):
\begin{equation}\label{eq:fs:params}
  \kappa_i:=\frac{1}{\gamma_i\nu^2},
  \
  r:=\frac{n\nu^2\tau^2}{\sigma^2}
\end{equation}
where $\kappa_1>\kappa_2>0$ are the effective penalty levels, and $r$ denotes the signal-to-noise ratio.

The rest of this section characterizes when Algorithm \ref{alg:nested_cv} improves on kCV, and by how much, in four steps. First, Lemma \ref{lem:cross} compares the expected test error of always selecting $\gamma_1$ with that of always selecting $\gamma_2$, and identifies the signal-to-noise threshold $r^\star$ at which these expected errors coincide. Second, Proposition \ref{prop:ref} shows that the same threshold governs the comparison between kCV and Algorithm \ref{alg:nested_cv} with isotropic covariances and as the number of observations per fold grows. Third, Theorem \ref{thm:phase} bounds the approximation error for finite samples and departures from isotropy: the same conclusion holds outside a band around $r^\star$ that widens as the covariances depart from isotropy or as the number of observations per fold decreases. Finally, Proposition \ref{prop:magnitude} quantifies the magnitude of Algorithm \ref{alg:nested_cv}'s improvement or deterioration over kCV in terms of parameters in the data-generation process.

The two dimensionless constants governing our analysis are:
\begin{equation}\label{eq:rstar}
    r^\star := \frac{2 + \kappa_1 + \kappa_2}{\kappa_1 + \kappa_2 + 2\kappa_1\kappa_2}, \
    s := \frac{(\kappa_1 - \kappa_2)(\kappa_1 + \kappa_2 + 2\kappa_1\kappa_2)}{(1+\kappa_1)^2(1+\kappa_2)^2},
\end{equation}
and arise from the following ridge test-error calculation of \citet[Section 4]{hoerl1970ridge}.

\begin{lemma}\label{lem:cross}
Suppose that the covariance matrices in the data-generation process \eqref{eqn:datagenridge}--\eqref{eq:fs:momb} are isotropic ($\bm{\Sigma}_x=\nu^2\bm{I}_p, \bm{\Sigma}_{\beta}=\tau^2\bm{I}_p$), and that $\bm{XX^\top}=n \nu^2 \bm{I}_p$. Then, for any $r > 0$, we have the identity
\begin{equation}\label{eq:candgap}
    \frac{\mathbb{E}[\Test(\gamma_1)- \Test(\gamma_2)| \bm{X}]}{p\nu^2\tau^2} = s\frac{r - r^\star}{r}.
\end{equation}
\end{lemma}
To keep the manuscript self-contained, we provide a proof of Lemma \ref{lem:cross} in Appendix \ref{app:cross}. 

We next compare Algorithm \ref{alg:nested_cv} with $k$-CV under the isotropy assumptions of Lemma \ref{lem:cross}, as the number of observations per fold grows. We introduce auxiliary normal random variables to represent the large-sample isotropic limit. With $m=n/k\to\infty$ and $r$ fixed, define
\begin{align}\label{eqn:zblockdef}
    \bm{Z}_j := \bm{B} + \sqrt{k/r} \bm{E}_j, \ j \in [k],
\end{align}
where $\bm{B},\bm{E}_1,\ldots,\bm{E}_k$ are independent $\mathcal{N}(\bm{0},\bm{I}_p)$ vectors. Here, $\bm{B}$ represents $\bm{\beta}_\star/\tau$, while $\sqrt{k/r}\bm{E}_j$ represents the noise in fold $j$. Applying both procedures to $\bm{Z}:=(\bm{Z}_1,\ldots,\bm{Z}_k)$ determines each procedure's limiting choice between $\gamma_1$ and $\gamma_2$, as defined in Appendix \ref{app:reference}. Let $D_{\mathrm{ref}}$ denote the event that $k$-CV chooses $\gamma_2$ while Algorithm \ref{alg:nested_cv} chooses $\gamma_1$, and let $\bar{\bm{Z}}:=\frac{1}{k}\sum_{j=1}^k\bm{Z}_j$. Define
\begin{equation}\label{eq:Pi}
    \Pi(r):=\frac{1}{p(1+1/r)}
    \mathbb{E}\left[
        \big\Vert\bar{\bm{Z}}\big\Vert_2^2 1_{D_{\mathrm{ref}}}
    \right].
\end{equation}
In the limit, the expected test-error difference between Algorithm \ref{alg:nested_cv} and $k$-CV equals that between $\gamma_1$ and $\gamma_2$ in Lemma \ref{lem:cross}, multiplied by the damping factor $\Pi(r)$ (proof deferred to Appendix \ref{app:reference}).

\begin{proposition}\label{prop:ref}
The quantity $\Pi(r)$ in \eqref{eq:Pi} is continuously differentiable on $(0,\infty)$ and satisfies $0 < \Pi(r) \leq 1$. Moreover, applying Algorithm \ref{alg:nested_cv} and $k$-fold cross-validation to $\bm{Z}$ yields the normalized expected test-error difference
\begin{equation}\label{eq:g}
    g(r) = s \frac{r - r^\star}{r} \Pi(r),
\end{equation}
which is negative for $r < r^\star$, zero at $r = r^\star$, and positive for $r > r^\star$.
\end{proposition}

To relate $g(r)$ to the relative performance of Algorithm \ref{alg:nested_cv} and $k$-fold CV, define the covariance distortion
\begin{equation}\label{eq:diso}
    \delta_{\mathrm{iso}} := \big\Vert (\bm{\Sigma}_x/\nu^2)^{1/2} - \bm{I}_p \big\Vert_{\mathrm{op}} + \big\Vert (\bm{\Sigma}_\beta/\tau^2)^{1/2} - \bm{I}_p \big\Vert_{\mathrm{op}},
\end{equation}
where $\Vert\cdot\Vert_{\mathrm{op}}$ denotes the spectral norm. In particular, $\delta_{\mathrm{iso}}=0$ when the covariance matrices are isotropic, as in Lemma \ref{lem:cross} and Proposition \ref{prop:ref}. Next, setting $n = km$, let
\begin{align}
G_m(r):=\frac{\mathbb{E}[\mathrm{Test}_{\mathrm{nested}}] - \mathbb{E}[\mathrm{Test}_{\mathrm{kCV}}]}{p\nu^2\tau^2}
\end{align}
be the normalized expected test-error gap at signal-to-noise ratio $r$, where $\Test_{k\mathrm{CV}}$ and $\Test_{\mathrm{nested}}$ denote the test errors under $k$-fold cross-validation and Algorithm \ref{alg:nested_cv} respectively.

We now bound the distance between $G_m$ and $g$ (proof deferred to Appendix \ref{app:coupling}).

\begin{theorem}\label{thm:phase}
Fix $p$, $k\geq3$, $\kappa_1>\kappa_2>0$, $\lambda_1>0$, and a compact interval $R\subset(0,\infty)$ containing $r^\star$. For every even $q\geq4$, there exists a constant $C_q<\infty$ depending only on $p$, $k$, $\kappa_1$, $\kappa_2$, $\lambda_1$, $R$, and $q$ such that
\begin{equation}\label{eq:approx}
    \sup_{r\in R}|G_m(r)-g(r)|
    \leq \varepsilon_m
    :=C_q
       \left(
          \delta_{\mathrm{iso}}
          +\sqrt{\frac{p+\log(2k)}{m}}
        \right)^{\frac{q-1}{2q+1}}.
\end{equation}
Moreover, define the relative transition radius $\rho_m:=\varepsilon_m/(s\inf_{r\in R}\Pi(r))$. Then, for $r \in R$, $G_m(r)<0$ for every $r<r^\star/(1+\rho_m)$, and $G_m(r)>0$ whenever $\rho_m<1$ and $r>r^\star/(1-\rho_m)$. Consequently, if $\rho_m<1$, every zero of $G_m$ in $R$ lies in the transition band
\begin{align*}
    R\cap\left[\frac{r^\star}{1+\rho_m},
              \frac{r^\star}{1-\rho_m}\right].
\end{align*}
In particular, under exact isotropy ($\delta_{\rm iso}=0$), we have 
$
\rho_m=O\left(m^{-\frac{q-1}{2(2q+1)}}\right),
$
and hence $\rho_m<1$ for all sufficiently large $m$.
\end{theorem}

Theorem \ref{thm:phase} shows that, for ridge regression, Algorithm \ref{alg:nested_cv} strictly improves on kCV below the band and is strictly worse above it. Moreover, the proof of Theorem \ref{thm:phase} reveals that the band is narrowest when the covariances are close to isotropic, the folds are large, and $\Pi(r)$ is not too small. Theorem \ref{thm:phase} does not, however, preclude several crossings within the transition band. Accordingly, Corollary~\ref{cor:single} (Appendix \ref{app:signs}) gives a sufficient condition for a single crossing point.

Next, we quantify the magnitude of the improvement from Algorithm \ref{alg:nested_cv} over kCV, by bounding the damping factor $\Pi$ in Proposition~\ref{prop:ref} (proof deferred to Appendix \ref{app:prop8proof}):
\begin{proposition}\label{prop:magnitude}
There exist constants $0<c_\Pi\leq C_\Pi<\infty$, depending on $p$, $k$, $\kappa_1$, and $\kappa_2$, such that
\begin{equation}\label{eq:pi-rate}
 c_\Pi(1+r)^{-(p/2+1)}
 \leq\Pi(r)\leq
 C_\Pi(1+r)^{-(p/2+1)}
 \qquad\text{for all }r>0.
\end{equation}
\end{proposition}

Combining Proposition \ref{prop:magnitude} with \eqref{eq:g} shows that the normalized reduction in test error from Algorithm \ref{alg:nested_cv} relative to $k$-CV grows at rate $r^{-1}$ as $r\downarrow0$, and the normalized increase in test error decays at rate $r^{-(p/2+1)}$ as $r\to\infty$. 

\textbf{Numerical Illustration:} We now illustrate the finite-sample behavior of both procedures under the Gaussian model \eqref{eqn:datagenridge}--\eqref{eq:fs:momb}. We set $n=p=k=20$, $\bm{\Sigma}_\beta=\bm{I}_p$, $\{\gamma_1,\gamma_2\}=\{0.1, 1000\}$, $\nu=\tau=1$, and $\Lambda=\{0,2\}$. We compare isotropic covariates, $\bm{\Sigma}_x=\bm{I}_p$, with a block-correlated design satisfying $$\bm{\Sigma}_x=\bm{I}_{p/2}\otimes\begin{pmatrix}1&0.8\\0.8&1\end{pmatrix}.$$ In this setting, the isotropic threshold defined in Equation \eqref{eq:rstar} as $n \rightarrow \infty$ is $r^\star=1.1976$. We set $\sigma^2=n/r$ and use $30$ values of $r$: 25 logarithmically spaced values of $r$ between $0.05$ and $50$, together with $r^\star$ and $r\in\{100,200,500,1000\}$. For each value, we average over $20{,}000$ Monte Carlo replications, using the same simulated sample and folds for both methods. We measure performance by the percentage reduction in average test-set error relative to $k$-CV, computed as $100\left(1-\widehat{\mathbb{E}}[\Test_{\mathrm{nested}}]/ \widehat{\mathbb{E}}[\Test_{k\mathrm{CV}}]\right)$, where $\widehat{\mathbb{E}}$ denotes the average over the Monte Carlo replications. Positive values favor Algorithm \ref{alg:nested_cv} over kCV. Together, both experiments took $557.1$ seconds using the computational setup described in Section \ref{sec:numres}.

\begin{figure}[!h]
\centering
\includegraphics[width=\textwidth]{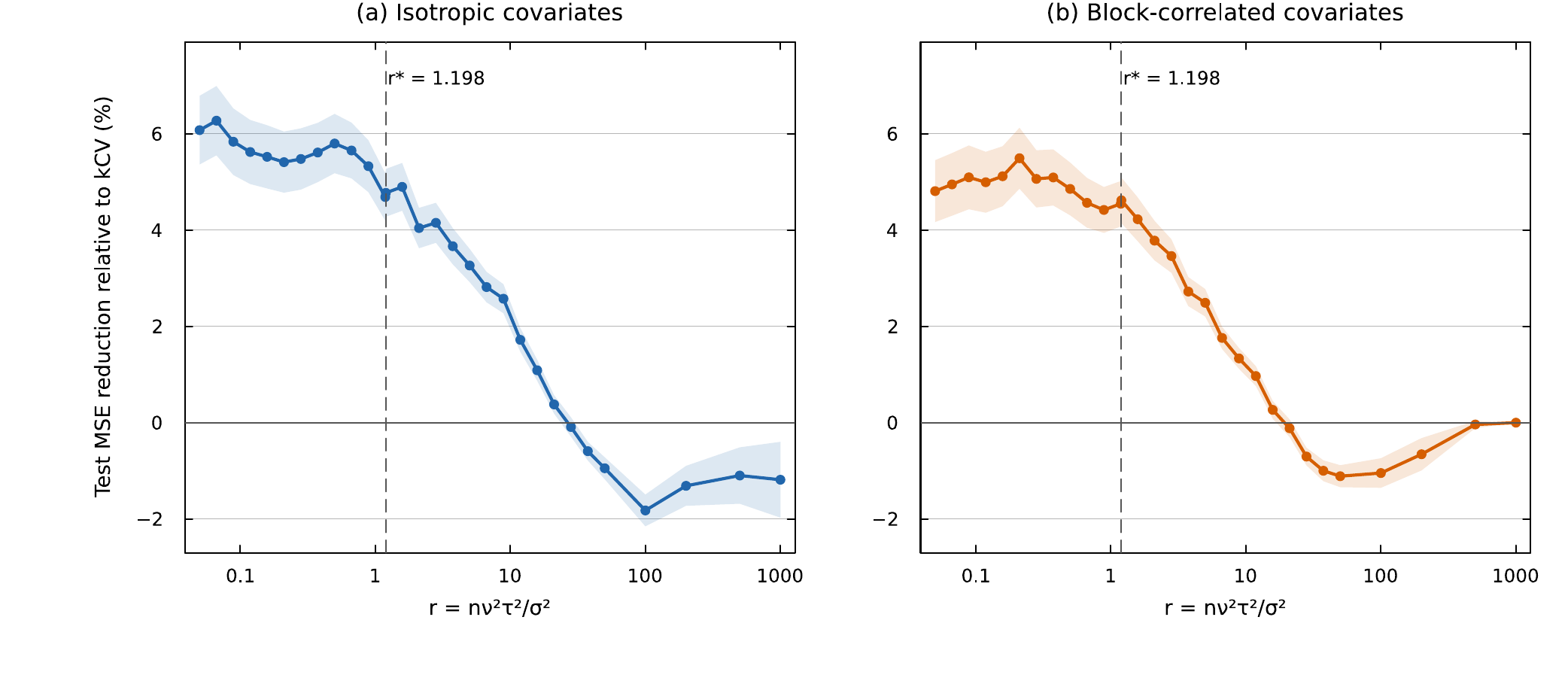}
\caption{Percentage reduction in average test-set error from Algorithm \ref{alg:nested_cv} relative to $k$-CV under (a) isotropic and (b) correlated covariates, for $n=p=k=20$. Shading gives pointwise $95\%$ confidence intervals over $20{,}000$ replications. Positive values favor Algorithm \ref{alg:nested_cv}. The dashed lines mark the isotropic threshold $r^\star$ as
$n\to\infty$. In particular, the crossing at $n=20\ll \infty$ need not occur at $r^\star$.
}
\label{fig:snrexperimentridge}
\end{figure}

Figure \ref{fig:snrexperimentridge} shows that Algorithm \ref{alg:nested_cv} improves out-of-sample performance for small values of $r$ by $0$---$6\%$ in both settings, with smaller losses ($0$---$2\%$) at larger $r$. The estimated improvement crosses zero near $r=27$ and $r=19$ in panels (a) and (b), respectively, above the isotropic limiting threshold $r^\star\approx1.20$. Because $r$ is not observable (and the data-generation process is not knowable) in practice, our results are most useful qualitatively. In particular, Algorithm \ref{alg:nested_cv} improves on $k$-fold cross-validation when the data are sufficiently noisy relative to the signal, but can worsen performance once the signal is sufficiently strong. 
\section{Numerical Experiments}\label{sec:numres}
We now evaluate Algorithm \ref{alg:nested_cv} on a suite of $13$ real-world datasets with sparse regression and CART as the lower-level learners. We implement all experiments in \verb|Julia| version $1.12$ and run them on a MacBook Pro with Apple M3 cores and $36$ GB RAM. We also use Mosek version 11.0 to solve the sparse-regularized conic subproblems described in Appendix \ref{ssec:sup_ridge_implementation}. Dataset sources and preprocessing details are given in Appendix \ref{ssec:dataset} and Table \ref{tab:datasets}.

\subsection{Sparse Regularized Regression With and Without Nesting}\label{ssec:ridge}
We benchmark Algorithm \ref{alg:nested_cv} on sparse regression for a suite of commonly studied real-world datasets. For a given sparsity level $\tau$, a given ridge parameter $\gamma>0$ and a given $\ell_1$ penalty parameter $\rho \geq 0$, the sparse-ridge learner solves 
\begin{align}
	\min_{\bm{\beta}\in \mathbb{R}^p}\quad & \frac{1}{2\gamma}\Vert \bm{\beta}\Vert_2^2+\rho \Vert \bm{\beta}\Vert_1+\frac{1}{2}\|\bm{X}\bm{\beta}-\bm{y}\|_2^2 \ \text{s.t.}\ \Vert \bm{\beta}\Vert_0\leq \tau.
\end{align} 

Thus, the hyperparameter vector is $\theta=(\tau,\rho, \gamma)$. We compare ordinary five-fold cross-validation, which selects $\theta$ by minimizing the five-fold CV error, against Algorithm \ref{alg:nested_cv}, which selects $\theta$ by minimizing the stability-regularized criterion in \eqref{eq:param_opt_2} after tuning the stability weight $\lambda$ by the outer nested loop. We solve the sparse-ridge problems approximately; we provide details of the solver, hyperparameter grids, and implementation choices in Appendix \ref{ssec:sup_ridge_implementation}.

For each dataset, we shuffle the observations and form 10 complementary train/test splits by holding out every 10th observation in turn. Ordinary kCV and the stability-regularized nested procedure are evaluated on the same splits. For each selected hyperparameter vector, we retrain on the full training/validation set and evaluate held-out test MSE on the corresponding test set.

As a benchmark, we implement repeated five-fold cross-validation as reviewed in Section \ref{ssec:litreview}, since averaging over repeated five-fold partitions can reduce the variance of cross-validation \citep{kim2009estimating}. In particular, we use the same cyclic coordinate descent procedure as ordinary kCV, but average each candidate's cross-validation error over twenty five-fold partitions: the original partition and nineteen additional random partitions, which remain fixed throughout the search. We then refit the selected model as above. We choose twenty repetitions to make the computational cost roughly comparable to Algorithm~\ref{alg:nested_cv}. This allows us to examine whether the additional computation required by Algorithm~\ref{alg:nested_cv} is worthwhile, or whether we would do better simply by repeating ordinary kCV.

As additional benchmarks, we also implement three alternatives to kCV reviewed in Section \ref{ssec:litreview}. Namely, the one-standard-error (1SE) rule \citep{breiman2017classification, hastie2009elements}, Estimation Stability with Cross-Validation (ESCV) \citep{limyu2016escv}, and Cross-Validation with Confidence (CVC) \citep{lei2020cross}, running all three methods on all hyperparameter combinations evaluated during ordinary kCV. The 1SE rule restricts attention to combinations whose cross-validation error is within one estimated standard error of the minimum obtained by kCV. We adapt ESCV to our three hyperparameters by selecting the combination with the lowest ES score among those satisfying $\tau\leq\tau_{\mathrm{kCV}}$, where the ES score measures the normalized variance of predictions across models fitted with one fold removed. For CVC, we perform a joint test comparing each combination with all other evaluated combinations, using a Gaussian multiplier bootstrap with $2{,}000$ draws and centering the loss differences within each fold. We retain combinations with $p$-values at least $0.05$. For both 1SE and CVC, we select the remaining combination with the smallest $\tau$, breaking ties by the largest $\rho$ and then the smallest $\gamma$. We use the same ordering to break ties in the ESCV score. Appendix~\ref{ssec:sup_ridge_implementation} gives the standard-error calculation used by the 1SE rule.

Table~\ref{tab:comparison_ourmethods} reports each dataset's dimensions and three measures for sparse ridge regression under ordinary cross-validation (kCV), nested stability-regularized cross-validation (nested-kCV), and repeated cross-validation (repeated-kCV): the empirical hypothesis stability ($\hat{\mu}_h$), average five-fold cross-validation error (``CV'') or nested cross-validation error (``nCV''), and average test set error (``MSE''). For each dataset, we then compute the ratio of the test-error metric under nested-kCV or repeated-kCV compared to that under ordinary kCV. We report percentage improvement as $100$ times one minus the geometric mean of these ratios, giving each dataset equal weight. Table \ref {tab:sparse_escv} provides equivalent results for the 1SE, ESCV, and CVC methods.

\begin{table*}[h]
\centering\footnotesize
\setlength{\tabcolsep}{3.5pt}
\begin{tabular}{@{}l r r r r r r r r r r r @{}}
\toprule
Dataset & $n$ & $p$
& \multicolumn{3}{c@{\hspace{0mm}}}{kCV}
& \multicolumn{3}{c@{\hspace{0mm}}}{nested-kCV}
& \multicolumn{3}{c@{\hspace{0mm}}}{repeated-kCV} \\
\cmidrule(l){4-6} \cmidrule(l){7-9} \cmidrule(l){10-12}
& & & $\hat{\mu}_h$ & CV & MSE
& $\hat{\mu}_h$ & nCV & MSE
& $\hat{\mu}_h$ & CV & MSE \\
\midrule
Wine & 6497 & 11 & 0.016 & 0.545 & 0.545 & 0.016 & 0.545 & 0.545 & 0.016 & 0.545 & 0.545 \\
Housing & 506 & 13 & 3.441 & 25.805 & 25.832 & 3.441 & 26.373 & 25.832 & 3.869 & 25.776 & 25.961 \\
AutoMPG & 392 & 25 & 2.638 & 9.289 & 9.167 & 2.571 & 9.532 & 9.204 & 2.858 & 9.626 & 9.176 \\
Hitters & 263 & 19 & 0.013 & 0.077 & 0.080 & 0.012 & 0.083 & 0.081 & 0.015 & 0.078 & 0.081 \\
Prostate & 97 & 8 & 0.101 & 0.528 & 0.530 & 0.092 & 0.557 & 0.543 & 0.111 & 0.533 & 0.542 \\
Servo & 167 & 19 & 0.222 & 0.811 & 0.810 & 0.209 & 0.859 & 0.816 & 0.275 & 0.776 & 0.758 \\
Toxicity & 38 & 9 & 0.021 & 0.038 & 0.062 & 0.014 & 0.057 & 0.067 & 0.016 & 0.043 & 0.066 \\
Steam & 25 & 8 & 0.214 & 0.452 & 0.590 & 0.173 & 0.598 & 0.509 & 0.226 & 0.473 & 0.659 \\
Alcohol2 & 44 & 21 & 0.042 & 0.211 & 0.301 & 0.045 & 0.270 & 0.230 & 0.040 & 0.220 & 0.267 \\
\midrule
TopGear & 242 & 373 & 0.021 & 0.044 & 0.043 & 0.021 & 0.052 & 0.043 & 0.024 & 0.043 & 0.046 \\
Bardet & 120 & 200 & 0.003 & 0.007 & 0.012 & 0.002 & 0.008 & 0.008 & 0.003 & 0.008 & 0.008 \\
Vessel & 180 & 486 & 0.008 & 0.018 & 0.025 & 0.007 & 0.021 & 0.023 & 0.006 & 0.021 & 0.021 \\
Riboflavin & 71 & 4088 & 0.136 & 0.265 & 0.207 & 0.135 & 0.294 & 0.207 & 0.120 & 0.247 & 0.230 \\
\bottomrule
\end{tabular}
\caption{
Average performance of $k$-fold cross-validation (kCV), nested stability-regularized cross-validation (nested-kCV, i.e., Algorithm~\ref{alg:nested_cv}), and repeated five-fold cross-validation (repeated-kCV), using sparse ridge regression as the learner. Datasets are sorted by how overdetermined they are ($n/p$), with the overdetermined and underdetermined cases separated. All entries are averaged over ten train/test splits. We report empirical hypothesis stability $\hat{\mu}_h$, the estimated test-set error (CV or nCV), and the held-out test-set error (MSE). 
}
\label{tab:comparison_ourmethods}
\end{table*}

\begin{table*}[h!]
\centering\footnotesize
\setlength{\tabcolsep}{3.5pt}
\begin{tabular}{@{}l r r r r r r r r r r r @{}}
\toprule
Dataset & $n$ & $p$ & \multicolumn{3}{c}{1SE} & \multicolumn{3}{c}{ESCV} & \multicolumn{3}{c}{CVC} \\
\cmidrule(l){4-6} \cmidrule(l){7-9} \cmidrule(l){10-12}
& & & $\hat{\mu}_h$ & CV & MSE & $\hat{\mu}_h$ & CV & MSE & $\hat{\mu}_h$ & CV & MSE \\
\midrule
Wine & 6497 & 11 & 0.049 & 0.562 & 0.562 & 0.011 & 0.565 & 0.566 & 0.015 & 0.547 & 0.548 \\
Housing & 506 & 13 & 4.449 & 27.458 & 27.683 & 2.844 & 36.605 & 36.243 & 3.872 & 26.736 & 27.281 \\
AutoMPG & 392 & 25 & 3.561 & 10.081 & 9.907 & 0.844 & 13.688 & 14.119 & 3.187 & 9.797 & 9.836 \\
Hitters & 263 & 19 & 0.020 & 0.080 & 0.081 & 0.008 & 0.085 & 0.088 & 0.021 & 0.080 & 0.081 \\
Prostate & 97 & 8 & 0.129 & 0.559 & 0.585 & 0.067 & 0.616 & 0.611 & 0.219 & 0.584 & 0.619 \\
Servo & 167 & 19 & 0.240 & 0.842 & 0.807 & 0.197 & 0.870 & 0.906 & 0.278 & 0.909 & 0.904 \\
Toxicity & 38 & 9 & 0.019 & 0.043 & 0.067 & 0.012 & 0.047 & 0.073 & 0.029 & 0.050 & 0.062 \\
Steam & 25 & 8 & 0.207 & 0.497 & 0.583 & 0.136 & 0.525 & 0.662 & 0.227 & 0.524 & 0.620 \\
Alcohol2 & 44 & 21 & 0.083 & 0.238 & 0.243 & 0.040 & 0.231 & 0.438 & 0.078 & 0.270 & 0.428 \\
\midrule
TopGear & 242 & 373 & 0.032 & 0.054 & 0.048 & 0.010 & 0.072 & 0.067 & 0.032 & 0.061 & 0.060 \\
Bardet & 120 & 200 & 0.003 & 0.008 & 0.013 & 0.003 & 0.008 & 0.012 & 0.006 & 0.011 & 0.013 \\
Vessel & 180 & 486 & 0.013 & 0.022 & 0.029 & 0.010 & 0.040 & 0.044 & 0.040 & 0.054 & 0.056 \\
Riboflavin & 71 & 4088 & 0.208 & 0.332 & 0.244 & 0.141 & 0.288 & 0.234 & 0.268 & 0.379 & 0.324 \\
\bottomrule
\end{tabular}
\caption{
Average performance of 1SE, ESCV, and CVC for sparse ridge regression on the same datasets and ten train/test splits as Table~\ref{tab:comparison_ourmethods}. We report the average empirical hypothesis stability $\hat{\mu}_h$, cross-validation error (CV), and held-out test error (MSE). Datasets are ordered as in Table~\ref{tab:comparison_ourmethods}.
}
\label{tab:sparse_escv}
\end{table*}

We observe that Algorithm \ref{alg:nested_cv} reduces the average test-set error by $5.6\%$ relative to $k$-fold cross-validation, with a $3.2\%$ reduction on overdetermined datasets and a $10.7\%$ reduction on underdetermined datasets. The selected empirical hypothesis stability is also $8.0\%$ lower on average, and the selected error estimates are also better calibrated. Namely, the kCV estimate is $12.3\%$ below the corresponding held-out test MSE, whereas the nested estimate is $5.8\%$ above the held-out test MSE of the model selected by Algorithm \ref{alg:nested_cv}. Finally, the two methods select the same $(\tau,\rho, \gamma)$ on $70.8\%$ of instances, and the outer loop selects a positive stability weight on $48.5\%$ of runs. 

On the other hand, repeated cross-validation attains an average test-set error $3.1\%$ lower than kCV but $2.6\%$ higher than Algorithm \ref{alg:nested_cv}. Relative to kCV, this corresponds to a $0.4\%$ increase on overdetermined datasets and a $10.6\%$ reduction on underdetermined datasets. Its average empirical hypothesis stability is also $1.0\%$ higher than that of ordinary kCV.

The other three benchmarks all perform worse than kCV: 1SE increases average test-set error by $4.6\%$ and average empirical hypothesis stability by 40.6\% relative to kCV. This can be explained as the 1SE approach selects the most parsimonious learner, but parsimony is not the same as stability or generalizability. The ESCV approach, as reported in Table \ref{tab:sparse_escv}, also performs worse than kCV, with an average test-set error 24.7\% higher than kCV, although it reduces the average empirical hypothesis stability by 27.5\% compared to kCV on average. CVC (see Table \ref{tab:sparse_escv}) also increases average test-set MSE by 20.6\% and empirical hypothesis stability by 61.7\% relative to kCV.

The geometric mean of the per-dataset average runtimes was $34.7$ seconds for kCV, $611.3$ seconds for nested-kCV, and $762.8$ seconds for repeated-kCV. Using the kCV solutions, running 1SE, ESCV, and CVC took an additional $<0.01$, $0.02$, and $0.034$ seconds. Finally, the geometric means of the per-dataset average sparsity of the regressors are $6.94$, $7.14$, $7.82$, $4.18$, $4.99$, and $3.54$ for kCV, nested-kCV, repeated-kCV, 1SE, ESCV, and CVC, respectively.

\subsection{Performance of CART With and Without Nesting}
We now compare Algorithm \ref{alg:nested_cv} with kCV using the Julia implementation of CART \citep{breiman2017classification} in the \verb|DecisionTree.jl| package as the lower-level learner. We iteratively optimize the five-fold and nested five-fold cross-validation error with respect to the tree depth and the minimum number of observations per leaf, both over the grid $\{1, \ldots, 10\}$, with the minimum leaf size initially fixed to $5$ for both approaches. We set the grid $\Lambda=\{0, 10^{-4}, 10^{-3}, 10^{-2}, \ldots, 10\}$. All remaining CART parameters take their default value in \verb|DecisionTree.jl|. We use the same train/test splitting and reporting procedure as Section \ref{ssec:ridge}.

Table \ref{tab:comparison_ourmethods_3} reports each dataset's dimensions and average performance over 10 train/test splits for ordinary kCV, Algorithm \ref{alg:nested_cv}, and repeated-kCV: cross-validation error (``CV'', or ``nCV'' for nested kCV), empirical hypothesis stability ($\hat{\mu}_h$), and test set error (``MSE'').

We also benchmark against repeated five-fold cross-validation \citep{kim2009estimating}, using twenty fixed fold partitions as in Section~\ref{ssec:ridge}. Further, we also implement the 1SE rule \citep{breiman2017classification, hastie2009elements}, ESCV \citep{limyu2016escv}, and CVC \citep{lei2020cross}, using the hyperparameter combinations evaluated during ordinary kCV. For 1SE and CVC, we select the smallest tree depth among the combinations allowed by each rule, breaking ties by the largest minimum leaf size. For ESCV, we select the combination with the lowest ES score among those with depth at most the kCV depth, breaking ties using the same ordering. We report these three benchmarks in Table~\ref{tab:cart_escv}.

\begin{table*}[h!]
\centering\footnotesize
\setlength{\tabcolsep}{3.5pt}
\begin{tabular}{@{}l r r r r r r r r r r r @{}}
\toprule
Dataset & $n$ & $p$
& \multicolumn{3}{c@{\hspace{0mm}}}{kCV}
& \multicolumn{3}{c@{\hspace{0mm}}}{nested-kCV}
& \multicolumn{3}{c@{\hspace{0mm}}}{repeated-kCV} \\
\cmidrule(l){4-6} \cmidrule(l){7-9} \cmidrule(l){10-12}
& & & $\hat{\mu}_h$ & CV & MSE
& $\hat{\mu}_h$ & nCV & MSE
& $\hat{\mu}_h$ & CV & MSE \\
\midrule
Wine & 6497 & 11 & 0.248 & 0.535 & 0.532 & 0.244 & 0.539 & 0.534 & 0.234 & 0.538 & 0.533 \\
Housing & 506 & 13 & 7.828 & 18.440 & 21.362 & 7.135 & 20.615 & 18.533 & 8.543 & 19.455 & 23.747 \\
AutoMPG & 392 & 25 & 8.178 & 15.182 & 15.525 & 7.435 & 16.829 & 16.551 & 8.173 & 15.883 & 15.103 \\
Hitters & 263 & 19 & 0.025 & 0.052 & 0.054 & 0.026 & 0.055 & 0.053 & 0.032 & 0.051 & 0.054 \\
Prostate & 97 & 8 & 0.421 & 0.746 & 0.898 & 0.408 & 0.888 & 0.933 & 0.406 & 0.763 & 0.795 \\
Servo & 167 & 19 & 0.118 & 0.234 & 0.387 & 0.146 & 0.280 & 0.280 & 0.134 & 0.342 & 0.319 \\
Toxicity & 38 & 9 & 0.042 & 0.087 & 0.102 & 0.037 & 0.100 & 0.083 & 0.041 & 0.091 & 0.094 \\
Steam & 25 & 8 & 0.526 & 1.109 & 1.199 & 0.551 & 1.431 & 1.604 & 0.572 & 1.158 & 1.461 \\
Alcohol2 & 44 & 21 & 0.407 & 0.916 & 0.766 & 0.357 & 0.858 & 0.699 & 0.356 & 0.922 & 0.773 \\
\midrule
TopGear & 242 & 373 & 0.045 & 0.091 & 0.089 & 0.048 & 0.083 & 0.079 & 0.045 & 0.083 & 0.093 \\
Bardet & 120 & 200 & 0.010 & 0.017 & 0.019 & 0.009 & 0.019 & 0.016 & 0.011 & 0.016 & 0.021 \\
Vessel & 180 & 486 & 0.066 & 0.130 & 0.140 & 0.055 & 0.093 & 0.123 & 0.064 & 0.099 & 0.204 \\
Riboflavin & 71 & 4088 & 0.356 & 0.821 & 1.095 & 0.330 & 0.701 & 0.979 & 0.410 & 0.821 & 1.023 \\
\bottomrule
\end{tabular}
\caption{
Average performance of CART across a suite of real-world datasets where the ground truth is unknown, sorted by how overdetermined the dataset is ($n/p$), and separated into the overdetermined and underdetermined cases. The three groups of columns correspond to ordinary $k$-fold cross-validation, nested stability-regularized cross-validation, and repeated five-fold cross-validation, respectively. All entries are averaged over ten train/test splits. For repeated-kCV, CV is the average cross-validation error over twenty five-fold partitions.
}
\label{tab:comparison_ourmethods_3}
\end{table*}

\begin{table*}[h!]
\centering\footnotesize
\setlength{\tabcolsep}{3.5pt}
\begin{tabular}{@{}l r r r r r r r r r r r @{}}
\toprule
Dataset & $n$ & $p$ & \multicolumn{3}{c}{1SE} & \multicolumn{3}{c}{ESCV} & \multicolumn{3}{c}{CVC} \\
\cmidrule(l){4-6} \cmidrule(l){7-9} \cmidrule(l){10-12}
& & & $\hat{\mu}_h$ & CV & MSE & $\hat{\mu}_h$ & CV & MSE & $\hat{\mu}_h$ & CV & MSE \\
\midrule
Wine & 6497 & 11 & 0.178 & 0.544 & 0.542 & 0.091 & 0.593 & 0.593 & 0.190 & 0.541 & 0.539 \\
Housing & 506 & 13 & 9.826 & 19.236 & 20.883 & 7.712 & 19.643 & 25.383 & 12.189 & 21.882 & 19.676 \\
AutoMPG & 392 & 25 & 8.792 & 15.941 & 16.565 & 8.817 & 21.544 & 22.697 & 9.772 & 17.318 & 18.811 \\
Hitters & 263 & 19 & 0.025 & 0.058 & 0.064 & 0.022 & 0.058 & 0.064 & 0.027 & 0.057 & 0.061 \\
Prostate & 97 & 8 & 0.516 & 0.810 & 0.817 & 0.442 & 0.741 & 0.799 & 0.583 & 0.865 & 0.797 \\
Servo & 167 & 19 & 0.132 & 0.248 & 0.495 & 0.201 & 0.984 & 0.998 & 0.411 & 0.451 & 0.570 \\
Toxicity & 38 & 9 & 0.041 & 0.097 & 0.078 & 0.041 & 0.094 & 0.073 & 0.048 & 0.099 & 0.078 \\
Steam & 25 & 8 & 0.597 & 1.154 & 1.786 & 0.490 & 1.224 & 1.836 & 0.902 & 1.510 & 2.094 \\
Alcohol2 & 44 & 21 & 0.670 & 1.046 & 1.140 & 0.417 & 1.387 & 3.117 & 1.219 & 2.275 & 2.647 \\
\midrule
TopGear & 242 & 373 & 0.060 & 0.096 & 0.096 & 0.043 & 0.087 & 0.102 & 0.063 & 0.093 & 0.095 \\
Bardet & 120 & 200 & 0.014 & 0.020 & 0.020 & 0.010 & 0.018 & 0.021 & 0.012 & 0.021 & 0.021 \\
Vessel & 180 & 486 & 0.071 & 0.110 & 0.195 & 0.061 & 0.152 & 0.148 & 0.097 & 0.156 & 0.190 \\
Riboflavin & 71 & 4088 & 0.359 & 0.749 & 0.934 & 0.319 & 0.810 & 0.805 & 0.423 & 0.809 & 1.023 \\
\bottomrule
\end{tabular}
\caption{
Average performance of 1SE, ESCV, and CVC for CART on the same datasets and ten train/test splits as Table~\ref{tab:comparison_ourmethods_3}. We report the average empirical hypothesis stability $\hat{\mu}_h$, cross-validation error (CV), and held-out test error (MSE). Datasets are ordered as in Table~\ref{tab:comparison_ourmethods_3}, with the overdetermined and underdetermined cases separated.
}
\label{tab:cart_escv}
\end{table*}

We observe a similar benefit from Algorithm \ref{alg:nested_cv} compared to kCV for CART. It improves the test-set MSE by 6.8\% on average, with a 4.2\% average improvement on overdetermined datasets and a 12.3\% average improvement on underdetermined datasets. The empirical hypothesis stability is also 4.2\% lower on average, and the two methods select different hyperparameters on $66.9\%$ of all instances, with the outer loop selecting a positive stability weight on $83.1\%$ of runs. The nested estimate is also better calibrated than ordinary kCV: the ordinary kCV estimate is 10.3\% below the corresponding test-set MSE, while the nested estimate is 0.2\% higher than the test-set MSE. 

On the other hand, repeated-kCV performs similarly to ordinary kCV: it obtains an average test-set MSE $2.7\%$ higher than kCV. It also increases the empirical hypothesis stability measure by $3.6\%$ relative to ordinary kCV, indicating similar stability.

The other three benchmarks perform moderately worse than kCV (Table~\ref{tab:cart_escv}). Indeed, while 1SE and our CVC implementation favor parsimonious models, this does not translate into better prediction or stability: their test-set MSEs increase by $9.7\%$ and $20.7\%$ relative to kCV, respectively, while their empirical hypothesis stability measures increase by $12.4\%$ and $45.4\%$. These results illustrate that more parsimonious learners need not be more stable or generalize better. ESCV, by contrast, reduces empirical hypothesis stability by $6.2\%$, but its test-set MSE increases by $27.2\%$, suggesting that improving stability alone does not necessarily improve out-of-sample performance.

Finally, the geometric mean of the per-dataset average runtimes was 0.16s for ordinary kCV, $6.79$s for Algorithm \ref{alg:nested_cv}, $3.26$s for repeated-kCV, and the additional times taken by 1SE, ESCV, and CVC after running kCV were $<0.01$s, $<0.01$s, and $0.02$s, respectively. The mean tree depths were $4.94$, $5.55$, $4.82$, $2.88$, $3.05$, and $2.33$ for kCV, nested-kCV, repeated-kCV, 1SE, ESCV, and CVC.
\FloatBarrier

\section{Conclusion}\label{sec:conclusion}
This work proposes a new approach to hyperparameter selection that reduces out-of-sample disappointment and improves out-of-sample performance: selecting hyperparameters that minimize a weighted sum of the cross-validation error and the empirical hypothesis stability, with the weight chosen by an outer nested cross-validation procedure. Across a suite of real-world datasets, our approach improves the out-of-sample MSE relative to ordinary $k$-fold cross-validation by $5.6\%$ on average for sparse regression and $6.8\%$ on average for CART. It also improves calibration of the selected model's error estimate: the estimate moves from 12.3\% below the test MSE to 5.8\% above it for sparse ridge regression, and from 10.3\% below to 0.2\% above it for CART.

\subsection*{Acknowledgements}
We used GPT-Astra 6.1 and Claude Fable 5.1 to assist with editing the manuscript and writing the code for the numerical experiments. The authors take full responsibility for the paper.

\ECSwitch
\ECHead{Supplementary Material}

\section{Proofs Omitted From Section \ref{sec:stability_adjustment}}\label{sec:proofs}
In this section, we provide all proofs omitted from Section \ref{sec:stability_adjustment}.
\subsection{Proof of Theorem \ref{thm:genbound}}\label{append.genboundproof}
 To prove Theorem \ref{thm:genbound}, we invoke an auxiliary lemma, which we state and prove in Section \ref{ssec:auxlemmas}. We first prove Theorem \ref{thm:genbound}:

\begin{proof} 
We establish the result by combining one deterministic inequality with two concentration arguments. The deterministic inequality bounds the overall test error by a weighted average of the test errors of the models with a fold deleted, weighted by fold size, plus the largest expected change in loss at a fresh observation when one fold is deleted. Hoeffding's inequality then relates the test errors of the models with a fold deleted to the held-out losses reported by cross-validation. Finally, the event defining $c_\Theta(\cdot)$ relates the expected change in loss to the empirical hypothesis stability. Throughout the proof, we let $m_j:=\vert \mathcal{N}_j\vert$ denote the size of the $j$th fold for simplicity. Moreover, we let $h_j(\bm{\theta})$ denote the total held-out loss on fold $j$, i.e.,
\begin{align*}
    h_j(\bm{\theta}):=\sum_{i \in \mathcal{N}_j}\ell\left(y_i, \bm{\beta}^{(\mathcal{N}_j)}(\bm{\theta}, \bm{x}_i)\right).
\end{align*}
 
By the definition of $\Delta_{j,\theta}$, for every observation $(x, y)$ and every fold $j \in [k]$, we have
\begin{align}
\ell\bigl(y, \bm{\beta}(\bm{\theta}, \bm{x})\bigr)
\le \ell\bigl(y, \bm{\beta}^{(\mathcal{N}_j)}(\bm{\theta}, \bm{x})\bigr)
+ \Delta_{j,\theta}(\bm{x}, y).
\end{align}
Therefore, by taking expectations over $(\bm{x}, y) \sim \mathcal{D}$, multiplying by $m_j/n$, summing over
$j$, and using $\sum_{j} m_j = n$, we obtain the deterministic inequality
\begin{equation}\label{eq:decompC}
\mathbb{E}_{(\bm{x},y) \sim \mathcal{D}}\bigl[\ell\bigl(y, \bm{\beta}(\bm{\theta}, \bm{x})\bigr)\bigr]
\le
\sum_{j \in [k]} \frac{m_j}{n}
\mathbb{E}_{(\bm{x},y) \sim \mathcal{D}}\bigl[\ell\bigl(y, \bm{\beta}^{(\mathcal{N}_j)}(\bm{\theta}, \bm{x})\bigr)\bigr]
+ \max_{j \in [k]}
\mathbb{E}_{(\bm{x},y) \sim \mathcal{D}}\bigl[\Delta_{j,\bm{\theta}}(\bm{x}, y)\bigr].
\end{equation}
Note that \eqref{eq:decompC} holds for every realized training sample and every $\theta$. Therefore, it remains to bound both terms on the right-hand side of \eqref{eq:decompC} in order to prove the result.

We next bound the first term in \eqref{eq:decompC}, the weighted average of the models' fold-deleted test errors. Conditional on the observations outside $\mathcal{N}_j$, the model $\bm{\beta}^{(\mathcal{N}_j)}(\bm{\theta})$ is fixed, while the observations in $\mathcal{N}_j$ are i.i.d. draws from $\mathcal{D}$ that are independent of it. Their losses lie in $[0, M]$ and average to $h_j(\bm{\theta}) / m_j$, and their conditional mean is the test error of that model. Hoeffding's inequality, in its one-sided form for the average of $m_j$ independent random variables with range $M$, therefore gives, for every $t > 0$ \citep[see, e.g.,][Theorem 2.8]{boucheron2013concentration},
\begin{equation}
\mathbb{P}\left\{
\mathbb{E}_{(\bm{x},y) \sim \mathcal{D}}\bigl[\ell\bigl(y, \beta^{(\mathcal{N}_j)}(\bm{\theta}, \bm{x})\bigr)\bigr]
- \frac{h_j(\bm{\theta})}{m_j} > t \right\}
\le \exp\left( - \frac{2 m_j t^2}{M^2} \right).
\end{equation}
Next, set
\begin{align}
t_j := M \sqrt{\frac{\log(2k|\Theta|/\delta)}{2 m_j}},
\end{align}
so that the right side equals $\delta / (2k|\Theta|)$. A union bound over the
$k$ folds and the $|\Theta|$ candidates then shows that, with probability at
least $1 - \delta/2$, simultaneously for every $\bm{\theta} \in \Theta$,
\begin{equation}\label{eq:foldsC}
\begin{aligned}
\sum_{j \in [k]} \frac{m_j}{n}
\mathbb{E}_{(\bm{x},y) \sim \mathcal{D}}\bigl[\ell\bigl(y, \beta^{(\mathcal{N}_j)}(\bm{\theta}, \bm{x})\bigr)\bigr]
&\le \frac{1}{n} \sum_{j \in [k]} h_j(\bm{\theta})
+ \frac{M}{n} \sqrt{\frac{\log(2k|\Theta|/\delta)}{2}}
\sum_{j \in [k]} \sqrt{m_j}\\[2pt]
&\le \frac{1}{n} \sum_{j \in [k]} h_j(\bm{\theta})
+ M \sqrt{\frac{k \log(2k|\Theta|/\delta)}{2n}},
\end{aligned}
\end{equation}
where the last inequality holds, since, by Cauchy-Schwarz, $\sum_{j} \sqrt{m_j} \le \sqrt{k \sum_{j} m_j} = \sqrt{kn}$.
 
Finally, we bound the second term in \eqref{eq:decompC}. The obstacle is that the loss differences are themselves random: each $\Delta_{j,\bm{\theta}}$ is constructed from models fitted to the observed sample, so an inequality proved for a fixed function does not apply to it. This is why the definition \eqref{eq:ctheta} involves a supremum over every training set the learning rule might receive: it ensures that whatever functions the observed sample produces, they are among those that the supremum covers.

Accordingly, fix a candidate $\bm{\theta}$ and a fold $j$, and compare the sample average
$\tfrac{1}{n} \sum_{i=1}^{n} \Delta_{j,\bm{\theta}}(\bm{x}_i, y_i)$, which is the $j$th
average in Equation~\eqref{eq:hypothesisStability2}, with its population counterpart
$\mathbb{E}_{(\bm{x},y) \sim \mathcal{D}}[\Delta_{j,\bm{\theta}}(\bm{x}, y)]$. If the sample average is at
least as large, the inequality
\begin{equation}\label{eq:perfoldC}
\mathbb{E}_{(\bm{x},y) \sim \mathcal{D}}\bigl[\Delta_{j,\bm{\theta}}(\bm{x}, y)\bigr]
\le M q_+\!\left(
\frac{1}{Mn} \sum_{i=1}^{n} \Delta_{j,\bm{\theta}}(\bm{x}_i, y_i),
c_\Theta(\delta/2) \right)
\end{equation}
holds because $q_+(a, c) \ge a$ by Lemma \ref{lem:envelope}(i). If instead the sample average
falls below its population counterpart, then the event applies and bounds the
relative entropy between the two by $c_\Theta(\delta/2)$; since the population
counterpart, normalized by $M$, then lies between the normalized sample
average and one, the definition of $q_+$ gives \eqref{eq:perfoldC} again.
Because $q_+(\cdot, c_\Theta(\delta/2))$ is nondecreasing, we may take the
maximum over $j$ on both sides of \eqref{eq:perfoldC} and move it inside the
first argument, and Equation \eqref{eq:hypothesisStability2} identifies the result as the empirical
hypothesis stability:
\begin{equation}\label{eq:transferC}
\max_{j \in [k]} \mathbb{E}_{(\bm{x},y) \sim \mathcal{D}}\bigl[\Delta_{j,\bm{\theta}}(\bm{x}, y)\bigr] \le M q_+\!\left( \frac{\hat{\mu}_h(\bm{\theta})}{M}, c_\Theta(\delta/2) \right).
\end{equation}
 
Finally, the events underlying \eqref{eq:foldsC} and \eqref{eq:transferC} each
fail with probability at most $\delta/2$, so both hold with probability at
least $1 - \delta$. On their intersection, substituting \eqref{eq:foldsC} and
\eqref{eq:transferC} into \eqref{eq:decompC} proves the theorem, simultaneously
for every $\bm{\theta} \in \Theta$. 

\end{proof}

\subsection{Properties of the Confidence Envelope}\label{ssec:auxlemmas}

\begin{lemma}\label{lem:envelope}
Fix $c \in [0, \infty)$ and let $q_+(a, c) := \sup\{ q \in [a, 1] : D_{\mathrm{kl}}(a \| q) \le c \}$. Then, we have the following properties:
\begin{itemize}
\item (i) For every $a \in [0, 1]$ we have $q_+(a, c) \ge a$, and $q_+(a, 0) = a$. For $a < 1$ the supremum is attained and $q_+(a, c) < 1$. If, in addition, $c > 0$, then $D_{\mathrm{kl}}(a \| q_+(a, c)) = c$. The map $a \mapsto q_+(a, c)$ is nondecreasing and continuous on $[0, 1]$.
\item (ii) The map $a \mapsto q_+(a, c)$ is concave on $[0, 1]$.
\item (iii) If $c > 0$ and $a \in (0, 1)$, then, writing $q := q_+(a, c)$, the map $a \mapsto q_+(a,c)$ is differentiable at $a$ with
\begin{align}
\frac{\partial q_+}{\partial a}(a, c)
= \frac{q (1 - q)}{q - a}
\log\!\left( \frac{q (1 - a)}{a (1 - q)} \right) > 0.
\end{align}
\end{itemize}
\end{lemma}

\begin{proof}
(i) The value $q = a$ is feasible, because $D_{\mathrm{kl}}(a \| a) = 0 \le c$, hence,
$q_+(a, c) \ge a$. If $c = 0$, no other value is feasible, since
$D_{\mathrm{kl}}(a \| q) > 0$ whenever $q \ne a$, so $q_+(a, 0) = a$.
 
Now, fix $a < 1$. On $[a, 1)$ the map $q \mapsto D_{\mathrm{kl}}(a \| q)$ is continuous and
strictly increasing, because
\begin{equation}\label{eq:dq}
\frac{\partial}{\partial q} D_{\mathrm{kl}}(a \| q) = \frac{q - a}{q (1 - q)} > 0
\qquad \text{for } q \in (a, 1),
\end{equation}
and it diverges as $q \uparrow 1$, since the term
$(1 - a) \log\frac{1 - a}{1 - q}$ does. The feasible set
$\{q \in [a, 1) : D_{\mathrm{kl}}(a \| q) \le c\}$ is therefore a closed interval
$[a, q^{*}]$ with $q^{*} < 1$; the supremum is attained at $q^{*}$, and when
$c > 0$ continuity and strict monotonicity give $D_{\mathrm{kl}}(a \| q^{*}) = c$.
 
For monotonicity, let $a \le a'$. If $q_+(a, c) \le a'$, then
$q_+(a, c) \le a' \le q_+(a', c)$ and we are done. Otherwise
$q := q_+(a, c) > a'$, and since
\begin{equation}\label{eq:da}
\frac{\partial}{\partial t} D_{\mathrm{kl}}(t \| q)
= \log\!\left( \frac{t (1 - q)}{q (1 - t)} \right) \le 0
\qquad \text{for } t \in (0, q),
\end{equation}
the map $t \mapsto D_{\mathrm{kl}}(t \| q)$ is nonincreasing on $[0, q]$, hence
$D_{\mathrm{kl}}(a' \| q) \le D_{\mathrm{kl}}(a \| q) \le c$. Thus $q$ is feasible for $a'$, and
$q_+(a', c) \ge q = q_+(a, c)$.
 
Continuity at $a = 1$ follows from $a \le q_+(a, c) \le 1$; when $c = 0$, it holds on all of $[0, 1]$ because $q_+(a, 0) = a$. Assume then that $c > 0$, fix $a \in [0, 1)$, and let $a_m \to a$ be a monotone sequence in $[0, 1)$. By monotonicity, $q_+(a_m, c)$ converges to some limit $L$ with $a \le L \le q_+(\max\{a_1, a\}, c) < 1$. Because $D_{\rm kl}$ is continuous in $[0, 1] \times (0,1)$, passing to the limit in $D_{\mathrm{kl}}(a_m \| q_+(a_m, c)) = c$ gives $D_{\mathrm{kl}}(a \| L)=c$, and hence $L = q_+(a, c)$ by the uniqueness established above.
 
(ii) The map $(a, q) \mapsto D_{\mathrm{kl}}(a \| q)$ is jointly convex, by the log-sum
inequality \citep[Theorem 2.7.2]{cover2006elements}. The set $K_c := \{(a, q) \in [0, 1]^2 : a \le q, \ D_{\mathrm{kl}}(a \| q) \le c\}$
is therefore convex because it is the intersection of a convex sublevel set and a half-plane. Let $a_1, a_2 \in [0, 1]$ and $t \in [0, 1]$. By
part~(i), $(a_1, q_+(a_1, c))$ and $(a_2, q_+(a_2, c))$ lie in $K_c$, so their
convex combination lies in $K_c$ as well; that is,
$t q_+(a_1, c) + (1 - t) q_+(a_2, c)$ is feasible for
$t a_1 + (1 - t) a_2$, hence
\begin{align*}
q_+\bigl(t a_1 + (1 - t) a_2, c\bigr)
\ge t q_+(a_1, c) + (1 - t) q_+(a_2, c).
\end{align*}
 
(iii) By part~(i), $q = q_+(a, c)$ satisfies $D_{\mathrm{kl}}(a \| q) = c$ with
$q \in (a, 1)$, and it is the only solution in $[a, 1)$ by strict
monotonicity. By \eqref{eq:dq}, $\partial_q D_{\mathrm{kl}}(a \| q) \ne 0$ at this point,
so the implicit function theorem applies to $D_{\mathrm{kl}}(a \| q) - c = 0$ and gives
local differentiability with
\begin{align*}
\frac{\partial q_+}{\partial a}(a, c)
= - \frac{\partial_a D_{\mathrm{kl}}(a \| q)}{\partial_q D_{\mathrm{kl}}(a \| q)}
= - \log\!\left( \frac{a (1 - q)}{q (1 - a)} \right) \cdot \frac{q (1 - q)}{q - a}
= \frac{q (1 - q)}{q - a}
\log\!\left( \frac{q (1 - a)}{a (1 - q)} \right),
\end{align*}
using \eqref{eq:dq} and \eqref{eq:da}. The expression is positive because
$a < q$ implies $q(1 - a) > a(1 - q)$.
\end{proof}

\subsection{A Family of Instances That Certify the Tightness of Theorem \ref{thm:genbound}}\label{ssec:worstcase}

We now provide a family of instances where Theorem \ref{thm:genbound} holds with equality:

\begin{proposition}
Fix $\delta\in(0,1)$ and integers $k\geq 2$ and $m\geq 1$ such that
\begin{align*}
m>\frac{1}{2}\log\!\left(\frac{2k}{\delta}\right),
\end{align*}
and set $n:=km$. Partition the sample into $k$ folds of size $m$, and let
\begin{align*}
t:=\sqrt{\frac{\log(2k/\delta)}{2m}}<1.
\end{align*}

Choose $q\in(0,1-t)$ and $r\in\{0,\ldots,n-1\}$ such that, for $B\sim\mathrm{Binomial}(n,q)$,
\begin{align*}
 a:=\frac rn<q,
 \
 \mathbb{P}\left(
        D_{\mathrm{kl}}\left(\frac{B}{n}\middle\Vert q\right)
        >D_{\mathrm{kl}}(a\Vert q)
    \right)
    \leq \frac{\delta}{2}
    <\mathbb{P}(B\leq r),
\end{align*}
and set $c:=D_{\mathrm{kl}}(a\|q)$.  There is a bounded-loss learning problem with $M=1$, a singleton candidate set, and equal folds for which $c_\Theta(\delta/2)=c$ is the smallest valid radius and Theorem \ref{thm:genbound} holds with equality on an event of positive probability. 
\end{proposition}

\begin{proof}
Let $\mathcal X=\{\bm{x}_{\rm cv},\bm{x}_{\rm st},\bm{x}_0\}$, let $Y=0$ almost surely, and set
\begin{align*}
 \mathbb{P}(\bm{X}=\bm{x}_{\rm cv})=t,
 \
 \mathbb{P}(\bm{X}=\bm{x}_{\rm st})=q,
 \
 \mathbb{P}(\bm{X}=\bm{x}_0)=1-t-q.
\end{align*}
Use $\ell(y,\widehat y)=|y-\widehat y|$ and $\Theta=\{\bm{\theta}\}$.  Training on all $n$ observations returns $\bm{\beta}(\bm{\theta},\bm{x})=\mathbf 1\{\bm{x}\in\{\bm{x}_{\rm cv},\bm{x}_{\rm st}\}\}$, whereas deleting any fold returns $\bm{\beta}^{(\mathcal N_j)}(\bm{\theta},\bm{x})=\mathbf 1\{\bm{x}=\bm{x}_{\rm cv}\}$. Hence, for every possible training set and fold,
\begin{align*}
 \Delta_{j,\bm{\theta}}(\bm{x},y)=\mathbf 1\{\bm{x}=\bm{x}_{\rm st}\}.
\end{align*}
Its population mean is $q$, and its empirical mean is the fraction $\widehat q$ of observations equal to $\bm{x}_{\rm st}$. Because this loss difference is the same for every possible training set and fold, the supremum in \eqref{eq:ctheta} does not alter the failure event. By assumption, the failure event at radius $c=D_{\mathrm{kl}}(a\Vert q)$ has probability at most $\delta/2$. If $0\leq c'<c$, then, since $u\mapsto D_{\mathrm{kl}}(u\Vert q)$ decreases on $[0,q]$, the failure event at radius $c'$ contains $\{\widehat q\leq a\}$, whose probability is $\mathbb{P}(B\leq r)>\delta/2$. Thus, $c$ is the smallest valid radius. Also, $q_+(a,c)=q$.

Let $\mathcal E$ be the event that the sample contains no observation with $X=x_{\rm cv}$ and exactly $r$ observations with $\bm{X}=\bm{x}_{\rm st}$.  Then 
\begin{align*}
 \mathbb{P}(\mathcal E)=\binom nr q^r(1-t-q)^{n-r}>0.
\end{align*}
On $\mathcal E$, every held-out loss is zero, so $h(\bm{\theta})=0$, while $\widehat\mu_h(\bm{\theta})=a$.  The full-sample model has population risk $t+q$.  Since $|\Theta|=1$ and the folds are equal, the final term in Theorem \ref{thm:genbound} is $t$. Therefore,
\begin{align*}
\mathbb{E}\left[\ell\bigl(Y,\bm{\beta}(\bm{\theta},\bm{X})\bigr)\right]=t+q=h(\bm{\theta})+q_+\left(\widehat\mu_h(\bm{\theta}),c\right)+\sqrt{\frac{k\log(2k/\delta)}{2n}},
\end{align*}
which is equality in Theorem \ref{thm:genbound}. Scaling the loss gives the same example for any $M>0$.
\end{proof}

\subsection{Proof of Corollary \ref{cor:affine}}\label{ssec:proofcorrbound}
\begin{proof}
We first establish the validity of the bound \eqref{eq:affine}, then show that the intercept $M b_c(\lambda)$ is indeed the smallest candidate-independent intercept, and conclude by showing that minimizing the right-hand side of the bound recovers \eqref{eqn:genbound}. 

To show the validity of the bound, let us fix $\lambda \ge 0$ and $a \in [0, 1]$. Then, the definition of $b_c$ as a supremum gives $q_+(a, c) \le \lambda a + b_c(\lambda)$, and
$b_c(\lambda) \le \sup_a q_+(a, c) = 1 < \infty$. Substituting
$a = \hat{\mu}_h(\bm{\theta})/M$ into \eqref{eqn:genbound}, on the event constructed in the
proof of Theorem~\ref{thm:genbound}, yields \eqref{eq:affine}. Moreover, the substitution is
deterministic, so the display holds for every $\bm{\theta}$ and every $\lambda$
simultaneously on that event.

 Next, to establish minimality, let us fix $\lambda$ and let $b < b_c(\lambda)$. By the definition of the supremum, there exists $a \in [0, 1]$ with
$q_+(a, c) - \lambda a > b$, that is,
$M q_+(a, c) > \lambda (M a) + M b$: with intercept $M b$, the affine right-hand side fails to dominate the stability correction of
\eqref{eqn:genbound} at the stability level $\hat{\mu}_h(\bm{\theta}) = M a$.
 
Finally, we show that the optimal $\lambda$ recovers the bound \eqref{eqn:genbound}. Specifically, we show that
$\inf_{\lambda \ge 0} \{ \lambda a + b_c(\lambda) \} = q_+(a, c)$ for every
$a \in [0, 1]$. Since the remaining terms of \eqref{eq:affine} do not involve $\lambda$, this proves the claim. First, the inequality $\inf_{\lambda \ge 0} \{ \lambda a + b_c(\lambda) \} \geq q_+(a, c)$ holds by the validity of the bound \eqref{eq:affine}. For $a \in (0, 1)$, let $\lambda_a$ be the derivative of $q_+(\cdot, c)$ at $a$
when $c > 0$, or $\lambda_a = 1$ when $c = 0$; in either case
$\lambda_a \ge 0$ is a supergradient, so
$q_+(a', c) \le q_+(a, c) + \lambda_a (a' - a)$ for all $a' \in [0, 1]$, hence $b_c(\lambda_a) = q_+(a, c) - \lambda_a a$ and
$\lambda_a a + b_c(\lambda_a) = q_+(a, c)$. For $a = 1$: by concavity the
chord slopes $(1 - q_+(a', c))/(1 - a')$ are nonnegative and nonincreasing in
$a'$, so the left derivative $\lambda_1 \ge 0$ exists and is a supergradient
at $a = 1$; the same argument applies. For $a = 0$: $b_c$ is nonincreasing
with $b_c(\lambda) \ge q_+(0, c)$ for every $\lambda$, obtained by taking
$a = 0$ in the supremum. Conversely, fix $\varepsilon > 0$ and, by continuity,
Lemma~\ref{lem:envelope}(i), pick $a' \in (0, 1)$ with
$q_+(a', c) \le q_+(0, c) + \varepsilon$. For
$\lambda \ge (1 - q_+(0, c))/a'$: if $a \ge a'$ then
$q_+(a, c) - \lambda a \le 1 - \lambda a' \le q_+(0, c)$, and if $a < a'$ then
$q_+(a, c) - \lambda a \le q_+(a', c) \le q_+(0, c) + \varepsilon$ by
monotonicity; hence $b_c(\lambda) \le q_+(0, c) + \varepsilon$. Letting
$\lambda \to \infty$ and $\varepsilon \downarrow 0$ gives
$\inf_{\lambda \ge 0} b_c(\lambda) = q_+(0, c)$.
\end{proof}

\subsection{Proof of Corollary \ref{cor:nested}}\label{ssec:proofofcorrnested}

\begin{proof}
The proof of Theorem~\ref{thm:genbound} uses only the following four properties of the family
$$\{\beta(\bm{\theta}), (\beta^{(\mathcal{N}_j)}(\bm{\theta}))_{j\in[k]} : \bm{\theta}\in\Theta\}.$$ First, that the losses lie in $[0,M]$. Second, that the candidate set and fold partition are finite and fixed independently of the data. Third, that each fold-deleted model is a deterministic function of the observations outside the deleted fold, which allows us to invoke the conditional Hoeffding step and the union bound over the $k|\Theta|$ candidate pairs. Fourth, the radius event, which allows us to apply the relative-entropy envelope step. 

The family
$\{\beta(\bm{\theta}^\star_\lambda), (\beta^{(\mathcal{N}_j)}(\bm{\theta}^\star_{j,\lambda}))_{j\in[k]} : \lambda\in\Lambda\}$ studied in Corollary \ref{cor:nested} satisfies the same four properties. Indeed, (ii) $\Lambda$ is finite and fixed, (i) and (iii) follow identically, and (iv) the equation \eqref{eq:cLambda} supplies the corresponding radius event. Repeating the argument of the proof of Theorem~\ref{thm:genbound} with $(\bm{\theta}, \Theta, c_\Theta)$ replaced by $(\lambda, \Lambda, c_\Lambda)$, the $k$-fold cross-validation error becomes $\bar{m}_\lambda$, the empirical hypothesis stability becomes $\hat{\mu}_{\rm meta}(\lambda)$, and the union bound runs over $k|\Lambda|$ pairs, which yields \eqref{eq:nestedbound}.
\end{proof}

\section{Proofs Omitted from Section \ref{sec:ridgeregression_theory}}\label{ssec:omittedproofsec3.1}
In this section, we provide all proofs omitted from Section \ref{sec:ridgeregression_theory}.

\subsection{Scale Analysis for the Ridge-Regression Model}
\begin{lemma}\label{lem:fs:norm} 
Let $\widetilde{\bm x}_i:=\bm x_i/\nu$, $\widetilde y_i:=y_i/(\nu\tau)$, and $\widetilde{\bm\beta}_\star:=\bm\beta_\star/\tau$. Moreover, let $\widetilde{\bm{\beta}}_{I}(\gamma)$ denote the ridge estimator in  \eqref{eq:fs:learner} computed from the rescaled data. Then the rescaled noise has variance $n/r$, the two effective penalties are $\kappa_1$ and $\kappa_2$, and, for every nonempty $I\subseteq[n]$, we have the equality
\begin{align*}
 \widehat{\bm\beta}_I(\gamma_\ell)
 =\tau\widetilde{\bm\beta}_I(1/\kappa_\ell).
\end{align*}
Every cross-validation score, empirical stability, and outer score is
$\nu^2\tau^2$ times its rescaled counterpart, and
\begin{align*}
 \mathrm{Test}(\bm b)-\sigma^2
 =\nu^2\tau^2
 \left(\frac{\bm b}{\tau}-\frac{\bm\beta_\star}{\tau}\right)^\top
 \frac{\bm\Sigma_x}{\nu^2}
 \left(\frac{\bm b}{\tau}-\frac{\bm\beta_\star}{\tau}\right),
\end{align*}
where $\mathrm{Test}(\bm b)$ denotes the expected squared prediction error of a fixed coefficient vector $\bm b$ on an independent draw from \eqref{eqn:datagenridge}. Consequently, the rescaling leaves every selection event unchanged.
\end{lemma}

\begin{proof}
Substitute $\bm x_i=\nu\widetilde{\bm x}_i$,
$y_i=\nu\tau\widetilde y_i$, and
$\bm\beta=\tau\widetilde{\bm\beta}$ in
\eqref{eqn:datagenridge} and \eqref{eq:fs:learner}. Then, we have the identity
\begin{align*}
    \frac{1}{2}\sum_{i \in I}(y_i-\bm{x}_i^\top \bm{\beta})^2+\frac{\vert I\vert}{2\gamma_l}\Vert \bm{\beta}\Vert_2^2=\nu^2 \tau^2\left[\frac{1}{2}\sum_{i \in I}(\widetilde{y}_i-\widetilde{\bm x}_i^\top \widetilde{\bm{\beta}})^2+\frac{\vert I\vert \kappa_l}{2}\Vert \widetilde{\bm{\beta}}\Vert_2^2\right],
\end{align*}
where we have used $1/(\gamma_l \nu^2)=\kappa_l$. This proves the claimed identity between the ridge estimators and shows that each squared residual is multiplied by $\nu^2 \tau^2$. Moreover, multiplying every candidate criterion by the same positive number does not change either its minimizers or its ties.
\end{proof}

Motivated by Lemma \ref{lem:fs:norm}, we set $\nu=\tau=1$ in the remaining proofs and restore the factor $\nu^2\tau^2$ in the statements.

\subsection{{Proof of Lemma \ref{lem:cross}}}\label{app:cross}
We note that Lemma \ref{lem:cross} can be proven by manipulating Equation (4.5) in \cite[]{hoerl1970ridge}, after translating their notation to our setting and rescaling. We provide an alternative proof to keep the manuscript self-contained. 

\begin{proof}
We first record an algebraic identity that we also use in Appendix \ref{app:reference}. Let $c_\ell := 1/(1+\kappa_\ell)$ and $\rho_r := r/(1+r)$. We claim that, for each $r > 0$,
\begin{equation}\label{eq:factor}
    \big( c_1 - \rho_r \big)^2 - \big( c_2 - \rho_r \big)^2 = s \frac{r - r^\star}{1+r}.
\end{equation}
Indeed, writing the left-hand side as $(c_1 - c_2)(c_1 + c_2 - 2\rho_r)$ and substituting the definitions of $c_1, c_2$ and $\rho_r$ gives
\begin{align*}
    (c_1 - c_2)(c_1 + c_2 - 2\rho_r) = \frac{(\kappa_2 - \kappa_1)}{(1+\kappa_1)(1+\kappa_2)} \cdot \frac{(2+\kappa_1+\kappa_2)(1+r) - 2r(1+\kappa_1)(1+\kappa_2)}{(1+\kappa_1)(1+\kappa_2)(1+r)},
\end{align*}
and the numerator of the second factor equals $(2+\kappa_1+\kappa_2) - r(\kappa_1 + \kappa_2 + 2\kappa_1\kappa_2)$, which is $-(\kappa_1+\kappa_2+2\kappa_1\kappa_2)(r - r^\star)$ by \eqref{eq:rstar}. This proves the equality \eqref{eq:factor}.
 
Next, since $n^{-1} \bm{XX^\top}=\bm{I}_p$ after setting $\nu=1$, solving \eqref{eq:fs:learner} with $I=[n]$ gives
\begin{align*}
 \widehat{\bm\beta}_{[n]}(\gamma_\ell)=c_\ell\left(\bm\beta_\star+\frac1n\bm X\bm\varepsilon\right).
\end{align*}
Here, conditional on $\bm{X}$, we have $\frac1n\bm X\bm\varepsilon\sim\mathcal N(\bm0,r^{-1}\bm I_p)$. Moreover, $\bm\varepsilon=(\varepsilon_1,\ldots,\varepsilon_n)^\top$, and the noise is independent of $\bm\beta_\star\sim\mathcal N(\bm0,\bm I_p)$, so
\begin{align*}
    \frac{\mathbb{E}[\Test(\gamma_1)- \Test(\gamma_2)|\bm{X}]}{p} = \frac{\kappa_1^2 + 1/r}{(1+\kappa_1)^2} - \frac{\kappa_2^2 + 1/r}{(1+\kappa_2)^2} .
\end{align*}
Since $\kappa_\ell = (1-c_\ell)/c_\ell$, each fraction equals $(1-c_\ell)^2 + c_\ell^2/r$. Their difference is
\begin{align*}
    (c_1 - c_2)\big[ (c_1 + c_2)(1 + 1/r) - 2 \big] = \big(1 + \tfrac{1}{r}\big) (c_1 - c_2)(c_1 + c_2 - 2\rho_r) = \big(1 + \tfrac{1}{r}\big) s \frac{r - r^\star}{1+r},
\end{align*}
where the last equality uses \eqref{eq:factor}. Thus, reimposing the factor $\nu^2\tau^2$ establishes the result.
\end{proof}

\subsection{Proof of Proposition \ref{prop:ref}}\label{app:reference}
We first formally define the event $D_{\mathrm{ref}}$ in Equation \eqref{eq:Pi}, and then prove Proposition \ref{prop:ref}. 

Let $\bm{Z} = (\bm{Z}_1, \ldots, \bm{Z}_k)$ be the Gaussian vector defined in Equation \eqref{eqn:zblockdef}, so that $\bm Z_j=\bm B+\sqrt{k/r}\bm E_j$, where $\bm B,\bm E_1,\ldots,\bm E_k$ are independent standard Gaussian vectors in $\mathbb{R}^p$. We now define how $k$-fold cross-validation and Algorithm \ref{alg:nested_cv} choose between $\gamma_1$ and $\gamma_2$ using $\bm{Z}$, in order to specify the disagreement event $D_{\text{ref}}$ in Equation \eqref{eq:Pi}. Next, let
\begin{equation}
 \bar{\bm Z}_{-j}:=\frac1{k-1}\sum_{q\ne j}\bm Z_q,
\end{equation}
and define
\begin{equation}\label{eq:HK}
 H_k(\kappa,\bm Z):=\frac1k\sum_{j=1}^k
 \left\|\bm Z_j-\frac{1}{1+\kappa}\bar{\bm Z}_{-j}\right\|_2^2.
\end{equation}

In \eqref{eq:HK}, block $j$ is predicted using the average of the remaining blocks multiplied by $(1+\kappa)^{-1}$. Next, $k$-fold cross-validation on $\bm Z$ selects
\begin{align*}
\hat\gamma^{\mathrm{ref}}_{k\mathrm{CV}}(\bm Z)
 \in\argmin_{\gamma_\ell\in\{\gamma_1, \gamma_2\}}{H_k(\kappa_\ell,\bm Z)}.
\end{align*}
On the other hand, for each outer fold $t$, define $H_{k-1}(\kappa,\bm Z_{-t})$ by applying \eqref{eq:HK} to the $k-1$ remaining blocks. Let $\ell_t(\bm Z)$ minimize $H_{k-1}(\kappa_\ell,\bm Z_{-t})$ over $\ell\in\{1,2\}$, and define
\begin{align}\label{eq:reference-outer}
 M_0(\bm Z)&:=\frac1k\sum_{t=1}^k
 \left\|\bm Z_t-\frac{1}{1+\kappa_{\ell_t(\bm Z)}}\bar{\bm Z}_{-t}\right\|_2^2,\notag\\
 M_1(\bm Z)&:=\frac1k\sum_{t=1}^k
 \left\|\bm Z_t-\frac{1}{1+\kappa_1}\bar{\bm Z}_{-t}\right\|_2^2.
\end{align}
For every fixed $\lambda_1>0$, minimizing the stability-regularized cross-validation criterion selects $\gamma_1$ with probability tending to one in the large-fold isotropic limit. After division by $\lambda_1\sqrt m$, the cross-validation difference vanishes, while the difference between the stability penalties has a strictly positive limit (see \eqref{eq:positive-margin} and the estimates following it in Appendix~\ref{app:coupling}). Accordingly, $M_1$ uses $\gamma_1$ on each set of training folds and does not depend on the value of $\lambda_1$. Hence, the nested method selects $\gamma_1$ if $M_1(\bm Z)<M_0(\bm Z)$ and otherwise retains the $k$-fold cross-validation choice. The event $D_{\mathrm{ref}}$ from Equation \eqref{eq:Pi} is therefore:
\begin{align}
 D_{\mathrm{ref}}=\big\{{H_k(\kappa_2,\bm Z)<H_k(\kappa_1,\bm Z)},\ M_1(\bm Z)<M_0(\bm Z)\big\}.
\end{align}

We now provide a formal proof of Proposition \ref{prop:ref}:
\begin{proof}
We first derive Equation \eqref{eq:g} by conditioning on $\bm{Z}$ and comparing the expected test errors at $\gamma_1$ and $\gamma_2$. We then prove $\Pi(r)>0$ by constructing an open subset of $D_{\mathrm{ref}}$, and establish continuous differentiability using the Gaussian density.

Write $c_\ell=(1+\kappa_\ell)^{-1}$ and $\rho_r=r/(1+r)$. By Equation \eqref{eqn:zblockdef}, conditional on $\bm{B}$, the vectors $\bm{Z}_j$ are independent with distribution $\mathcal{N}(\bm{B}, \frac{k}{r}\bm{I}_p)$. Therefore, combining their joint density with the density of $\bm{B}$ and completing the square gives
\begin{align*}
    \bm{B} | \bm{Z} \sim \mathcal{N}(\rho_r \bm{\bar Z}, (1+r)^{-1}\bm{I}_p),
\end{align*}
so $\mathbb E[\bm B\mid\bm Z]=\rho_r\bar{\bm Z}$. 

The first inequality defining $D_{\mathrm{ref}}$ is $H_k(\kappa_2,\bm Z)<H_k(\kappa_1,\bm Z)$, so $k$-fold cross-validation selects $\gamma_2$ by \eqref{eq:HK}. The second is $M_1(\bm Z)<M_0(\bm Z)$ in \eqref{eq:reference-outer}, so Algorithm~\ref{alg:nested_cv}  selects the positive stability weight, which yields $\gamma_1$ in this limit. Outside $D_{\mathrm{ref}}$, cross-validation already selects $\gamma_1$ or Algorithm~\ref{alg:nested_cv}  retains the cross-validation choice. Expanding the squared errors of $c_1\bar{\bm Z}$ and $c_2\bar{\bm Z}$ and taking conditional expectations therefore gives
\begin{align}\label{eq:reference-risk}
 g(r)&=\frac{(c_1-c_2)(c_1+c_2-2\rho_r)}p
       \mathbb E\left[\|\bar{\bm Z}\|_2^2 1_{D_{\mathrm{ref}}}\right]\notag\\
 &=s\frac{r-r^\star}{r}\Pi(r),
\end{align}
where the second equality uses \eqref{eq:factor} and \eqref{eq:Pi}. Also, $\mathbb E\|\bar{\bm Z}\|_2^2=p(1+1/r)$, so $0\leq\Pi(r)\leq1$.

To prove strict positivity, take $\bm Z_1=\cdots=\bm Z_{k-1}=\bm v$ and $\bm Z_k=(1-x)\bm v$, where $\bm v\in\mathbb R^p$ is a unit vector, and put $d=c_1+c_2\in(0,2)$. This will allow us to choose $x$ just beyond the point at which the inner comparison for $t<k$ switches to $\gamma_1$, while cross-validation on all $k$ blocks still selects $\gamma_2$. Define
\begin{equation}\label{eq:ref:fK}
 f(x):=(2-d)(k-1-2x)-\frac{d x^2}{k-2}.
\end{equation}
This quadratic is strictly decreasing for $x\geq0$, since $f'(x)=-2(2-d)-2dx/(k-2)<0$. Its positive zero $x_0$ lies in $((2-d)/2,(k-1)/2)$ because
\begin{equation*}
 f((2-d)/2)=(2-d)\left[k-3+d\left(1-\frac{2-d}{4(k-2)}\right)\right]>0>f((k-1)/2).
\end{equation*}
Substitution in \eqref{eq:HK} gives, for $t<k$,
\begin{equation*}
 H_{k-1}(\kappa_2,\bm Z_{-t})-H_{k-1}(\kappa_1,\bm Z_{-t})
 =-\frac{c_2-c_1}{k-1}f(x),
\end{equation*}
whereas at $x=x_0$,
\begin{equation}\label{eq:ref:score-sign}
 H_k(\kappa_2,\bm Z)-H_k(\kappa_1,\bm Z)
 =-\frac{c_2-c_1}{k}\left[2-d+\frac{d x_0^2}{(k-1)(k-2)}\right]<0.
\end{equation}
Choose $x\in(x_0,k)$ sufficiently close to $x_0$. Cross-validation on all $k$ blocks then selects $\gamma_2$, while each inner comparison with $t<k$ selects $\gamma_1$. Omitting the last block leaves identical blocks, with errors $(1-c_\ell)^2$, so this inner comparison selects $\gamma_2$. The first $k-1$ terms in \eqref{eq:reference-outer} cancel, leaving
\begin{equation*}
 M_1-M_0=\frac1k\big[(1-x-c_1)^2-(1-x-c_2)^2\big]
 =\frac{c_2-c_1}{k}(2-d-2x)<0.
\end{equation*}
Thus, this point belongs to $D_{\mathrm{ref}}$ and has nonzero mean $(1-x/k)\bm v$. The strict inequalities for $H_k, H_{k-1}$ and $M_1-M_0$ continue to hold in an open neighborhood, where the mean remains nonzero. The Gaussian density is positive everywhere, and hence \eqref{eq:Pi} gives $\Pi(r)>0$.

Finally, the inequalities defining $D_{\mathrm{ref}}$ do not depend on $r$. Let $\varphi_r$ denote the Gaussian density of $\bm Z$. Applying the multivariate Gaussian density formula \cite[Equation~(2.43)]{bishop2006} to \eqref{eqn:zblockdef} and differentiating gives
\begin{equation}\label{eq:ref:density-derivative}
 \partial_r\varphi_r(\bm Z)
 =\left[\frac{kp}{2r}-\frac{p}{2(1+r)}
 -\frac1{2k}\sum_{j=1}^k\|\bm Z_j-\bar{\bm Z}\|_2^2
 -\frac{\|\bar{\bm Z}\|_2^2}{2(1+r)^2}\right]\varphi_r(\bm Z).
\end{equation}
On each compact interval in $(0, \infty)$, the covariance eigenvalues $k/r$ and $k+k/r$ lie between two positive constants. Thus, $\varphi_r(\bm Z)\leq C e^{-c\|\bm Z\|_2^2}$ uniformly on that interval. Combining this with \eqref{eq:ref:density-derivative} and $\|\bar{\bm Z}\|_2^2\leq\|\bm Z\|_2^2/k$ gives an integrable bound
\begin{equation}\label{eq:ref:density-bound}
 \|\bar{\bm Z}\|_2^2\big(|\varphi_r(\bm Z)|+|\partial_r\varphi_r(\bm Z)|\big)
 \leq C(1+\|\bm Z\|_2^4)e^{-c\|\bm Z\|_2^2},
\end{equation}
where $C,c>0$ depend on the interval.
For $r$ in its interior, the mean value theorem bounds
$\|\bar{\bm Z}\|_2^2|\varphi_{r+h}(\bm Z)-\varphi_r(\bm Z)|/|h|$
by the right-hand side of \eqref{eq:ref:density-bound} for all sufficiently small $h$. 

Writing the expectation in \eqref{eq:Pi} as an integral against $\varphi_r$, we may differentiate under the integral by dominated convergence. Differentiating the factor $r/[p(1+r)]$ as well gives
\begin{equation}\label{eq:ref:Pi-derivative}
 \Pi'(r)=\frac{\Pi(r)}{r(1+r)}
 +\frac{r}{p(1+r)}\int_{D_{\mathrm{ref}}}
   \|\bar{\bm Z}\|_2^2\partial_r\varphi_r(\bm Z)\,d\bm Z.
\end{equation}
The derivative in \eqref{eq:ref:density-derivative} is continuous in $r$ for each $\bm Z$. Applying dominated convergence once more with \eqref{eq:ref:density-bound} shows that the integral in \eqref{eq:ref:Pi-derivative} is continuous in $r$. Thus $\Pi\in C^1(0,\infty)$. The sign claims follow from $s>0$ and $\Pi(r)>0$.
\end{proof}

\subsection{Proof of Theorem \ref{thm:phase}}\label{app:coupling}
In this section, we prove Theorem \ref{thm:phase}. By Lemma~\ref{lem:fs:norm}, we may set $\nu=\tau=1$ t simplify notation and restore the factor $\nu^2\tau^2$ at the end. 

In this and the following subsection, write
$\Pi_R:=\inf_{r\in R}\Pi(r)>0$ and, whenever $\rho_m<1$, let
\begin{align*}
    B_m:=R\cap\left[\frac{r^\star}{1+\rho_m},
                   \frac{r^\star}{1-\rho_m}\right].
\end{align*}

To establish Theorem \ref{thm:phase}, we first state two technical lemmas (Lemma~\ref{lem:empirical-stability} and Lemma~\ref{lem:gaussian-agreement}), combine them for the full result, and then prove both technical lemmas. In particular, Lemma~\ref{lem:empirical-stability} gives the approximation \eqref{eq:coup:stabapprox} for the rescaled empirical stability $\hat{\mu}^{(m)}_{\ell, S}/\sqrt{m}$, allowing the fitted coefficients to depend on the sample. Lemma~\ref{lem:gaussian-agreement} then bounds the candidate test-error differences in \eqref{eq:risk-coupling}, and, in \eqref{eq:decision-mismatch-rate}, the probability that Algorithm~\ref{alg:nested_cv} or $k$-fold cross-validation selects different ridge parameters in the finite sample and Gaussian models. 

We first introduce additional notation that is needed to establish all three results. First, write $\bm L_x=\bm\Sigma_x^{1/2}$ and $c_\ell=(1+\kappa_\ell)^{-1}$, and for each fold $j$, let
\begin{equation}\label{eq:fold-summaries}
 \bm A_j:=\frac1m\sum_{i\in\mathcal N_j}\bm x_i\bm x_i^\top,
\end{equation}
and
\begin{equation}\label{eqn:fold_u_defn}
     \bm u_j:=\frac1m\sum_{i\in\mathcal N_j}\bm x_i y_i.
\end{equation}
For a nonempty set of folds $S$, averages over $S$ are denoted by $\bar{\bm A}_S$, $\bar{\bm u}_S$ and $\bar{\bm Z}_S$. By \eqref{eq:fs:learner}, the coefficient fitted on these folds is
\begin{equation*}
 \widehat{\bm b}_{\ell,S}
   =(\bar{\bm A}_S+\kappa_\ell\bm I_p)^{-1}\bar{\bm u}_S.
\end{equation*}
The sets on which Algorithm~\ref{alg:nested_cv}  selects a ridge parameter are
$\mathfrak S:=\{[k]\}\cup\{[k]\setminus\{t\}:t\in[k]\}$.
For $S\in\mathfrak S$, write $\mathcal I_S=\bigcup_{j\in S}\mathcal N_j$ and let
\begin{equation}\label{eq:subset-stability}
 \widehat\mu_{\ell,S}^{(m)}
 :=\max_{j\in S}\frac1{|S|m}\sum_{i\in\mathcal I_S}
 \left|(y_i-\bm x_i^\top\widehat{\bm b}_{\ell,S\setminus\{j\}})^2
       -(y_i-\bm x_i^\top\widehat{\bm b}_{\ell,S})^2\right|
\end{equation}
be the empirical stability used in Algorithm~\ref{alg:nested_cv}. We then have the following result:
\begin{lemma}\label{lem:empirical-stability}
For every $t\geq1$, there exists a constant $C_t < \infty$, depending only on $p, k, \kappa_1, \kappa_2, t$ and $R$ such that, for every $S\in\mathfrak S$ and $\ell\in\{1,2\}$,
\begin{equation}\label{eq:coup:stabapprox}
 \left\|\frac{\widehat\mu_{\ell,S}^{(m)}}{\sqrt m}
 -\frac{4\sqrt{k}}{\pi\sqrt{r}}\max_{j\in S}\left\|\bm L_x
       (\widehat{\bm b}_{\ell,S}-\widehat{\bm b}_{\ell,S\setminus\{j\}})
       \right\|_2\right\|_{L_t}
 \leq C_t m^{-1/2}.
\end{equation}
\end{lemma}

Next, define the quantity
\begin{equation*}
d_m:=\delta_{\mathrm{iso}}+\sqrt{\frac{p+\log(2k)}m}.
\end{equation*}

Further, define the differences between the test errors at $\gamma_1$ and $\gamma_2$ in the finite sample and Gaussian model, respectively, by
\begin{align}\label{eq:finite-candidate-gap}
 \Delta_m&:=\mathrm{Test}(\gamma_1)-\mathrm{Test}(\gamma_2),\notag\\
 \Delta_\infty&:=\|c_1\bar{\bm Z}-\bm B\|_2^2
                 -\|c_2\bar{\bm Z}-\bm B\|_2^2.
\end{align}

We then have the following lemma:

\begin{lemma}\label{lem:gaussian-agreement}
For every even $q\geq4$, there exists a constant $C_q < \infty$, depending only on $p, k, \kappa_1, \kappa_2, \lambda_1, q$ and $R$, such that the finite sample and the Gaussian variables in \eqref{eqn:zblockdef} can be constructed together to satisfy
\begin{equation}\label{eq:risk-coupling}
 \|\Delta_m-\Delta_\infty\|_{L_1}\leq C_qd_m,
 \qquad \|\Delta_m\|_{L_q}\leq C_q.
\end{equation}
Let $\mathcal E_m$ be the event that the finite-sample and Gaussian selections agree for Algorithm~\ref{alg:nested_cv} and also for $k$-fold cross-validation. Then
\begin{equation}\label{eq:decision-mismatch-rate}
 \mathbb P(\mathcal E_m^c)\leq C_qd_m^{q/(2q+1)}.
\end{equation}
\end{lemma}

We now prove Theorem \ref{thm:phase}, and then prove the two intermediate lemmas.
\begin{proof}
We use Lemma~\ref{lem:gaussian-agreement} to control $\Delta_m-\Delta_\infty$ and $\mathbb P(\mathcal E_m^c)$, which together bound $|G_m(r)-g(r)|$. Let $R_m^{\mathrm{sel}}$ be the test error of Algorithm~\ref{alg:nested_cv}'s full-sample fit minus that of the full-sample fit selected by $k$-fold cross-validation. On $\mathcal E_m$, the Gaussian versions of Algorithm~\ref{alg:nested_cv} and $k$-fold cross-validation disagree only on $D_{\mathrm{ref}}$, so $R_m^{\mathrm{sel}}=\Delta_m1_{D_{\mathrm{ref}}}$. On $\mathcal E_m^c$, both $R_m^{\mathrm{sel}}$ and $\Delta_m1_{D_{\mathrm{ref}}}$ belong to $\{0,\Delta_m,-\Delta_m\}$, giving
\begin{equation*}
 |R_m^{\mathrm{sel}}-\Delta_m1_{D_{\mathrm{ref}}}|
 \leq2|\Delta_m|1_{\mathcal E_m^c}.
\end{equation*}
Proposition~\ref{prop:ref} gives $g(r)=p^{-1}\mathbb E[\Delta_\infty1_{D_{\mathrm{ref}}}]$. Hence H\"older's inequality and Lemma~\ref{lem:gaussian-agreement} imply, for $d_m\leq1$,
\begin{align*}
 |G_m(r)-g(r)|
 &\leq\frac1p\|\Delta_m-\Delta_\infty\|_{L_1}
      +\frac2p\|\Delta_m\|_{L_q}\mathbb P(\mathcal E_m^c)^{1-1/q}\\
 &\leq C_q\left[d_m+d_m^{(q-1)/(2q+1)}\right]
 \leq 2C_qd_m^{(q-1)/(2q+1)}.
\end{align*}

Absorbing the factor $2$ into $C_q$ proves Equation \eqref{eq:approx} when $d_m\leq1$. On the other hand, for $d_m > 1$, \eqref{eq:risk-coupling} bounds $\mathbb{E}{[|\Delta_m|]}$ independently of $m$, and the Gaussian expression for $\Delta_\infty$ in \eqref{eq:finite-candidate-gap} gives the same conclusion for $\mathbb{E}{[|\Delta_\infty|]}$. Since $|R_m^{\mathrm{sel}}|\leq|\Delta_m|$, these bounds give \eqref{eq:approx} in this case as well. Restoring $\nu$ and $\tau$ multiplies every test-error gap by $\nu^2\tau^2$, which cancels in $G_m$.

Finally, if $r\in R$ and $r<r^\star/(1+\rho_m)$, then
\begin{equation*}
 g(r)=-s(r^\star/r-1)\Pi(r)<-s\rho_m\Pi_R=-\varepsilon_m,
\end{equation*}
so $G_m(r)<0$. If $\rho_m<1$ and $r>r^\star/(1-\rho_m)$, the same calculation gives $g(r)>\varepsilon_m$ and hence $G_m(r)>0$. These signs locate every zero of $G_m$ in $B_m$. Under exact isotropy, \eqref{eq:approx} gives $\rho_m=O(m^{-(q-1)/(2(2q+1))})$.

\end{proof}

We now provide a formal proof of Lemma~\ref{lem:empirical-stability}:
\begin{proof}
We first control the empirical average of $|\bm x_i^\top\bm v|\,|\varepsilon_i|$ uniformly over directions $\bm v$. We then apply this bound to $\widehat{\bm b}_{\ell,S}-\widehat{\bm b}_{\ell,S\setminus\{j\}}$, even though these fitted coefficients depend on the sample.

Write $\varepsilon_i=\sqrt{km/r}\,\eta_i$, with independent standard normal $\eta_i$. By Jensen's inequality and Gaussian moments, the matrices $\bm A_j$ have uniformly bounded moments. Conditional on the covariates, $\bm u_j-\bm A_j\bm\beta_\star$ is Gaussian with covariance $(k/r)\bm A_j$, so $\bm u_j$ also has uniformly bounded moments. Since $\|(\bar{\bm A}_S+\kappa_\ell\bm I_p)^{-1}\|_{\mathrm{op}}\leq\kappa_\ell^{-1}$, we obtain
\begin{equation}\label{eq:coup:bmom}
 \max_{\ell,S\ne\varnothing}\|\widehat{\bm b}_{\ell,S}\|_{L_t}\leq C_t.
\end{equation}
Let $N_S=|S|m$. Symmetrization and contraction for $u\mapsto|u|$, using the moment function $u\mapsto u^t$ \cite[Lemma~6.3 and Theorem~4.12]{ledoux1991}, give
\begin{align}\label{eq:coup:emplinear}
 \left\|\sup_{\|\bm v\|_2=1}
 \left|\frac1{N_S}\sum_{i\in\mathcal I_S}|\bm x_i^\top\bm v|\,|\eta_i|
       -\frac2\pi\|\bm L_x\bm v\|_2\right|\right\|_{L_t}
 &\leq\frac{4}{N_S}
 \left\|\left\|\sum_{i\in\mathcal I_S}\xi_i|\eta_i|\bm x_i\right\|_2\right\|_{L_t}\notag\\
 &\leq C_t m^{-1/2},
\end{align}
where the $\xi_i$ are independent Rademacher variables. The expectation being subtracted is $(2/\pi)\|\bm L_x\bm v\|_2$ by independence. For the last inequality, condition on $\xi_i$ and $\eta_i$: the sum is Gaussian with covariance $(\sum_i\eta_i^2)\bm\Sigma_x$, and Gaussian and chi-square moments give the bound.

Fix $j\in S$ and put
\begin{align*}
 \bm d&:=\widehat{\bm b}_{\ell,S}-\widehat{\bm b}_{\ell,S\setminus\{j\}},\\
 \bm a&:=2\bm\beta_\star-\widehat{\bm b}_{\ell,S}
                         -\widehat{\bm b}_{\ell,S\setminus\{j\}}.
\end{align*}
The absolute difference of squared losses in \eqref{eq:subset-stability} is $|\bm x_i^\top\bm d|\,|2\varepsilon_i+\bm x_i^\top\bm a|$. Denote the average over $i\in\mathcal I_S$ by $T_{j,\ell,S}^{(m)}$. The triangle and Cauchy--Schwarz inequalities give
\begin{equation}\label{eq:coup:stab-remainder}
 \left|\frac{T_{j,\ell,S}^{(m)}}{\sqrt m}
       -2\sqrt{k/r}\frac1{N_S}\sum_{i\in\mathcal I_S}|\bm x_i^\top\bm d|\,|\eta_i|\right|
 \leq\frac{\|\bar{\bm A}_S\|_{\mathrm{op}}\|\bm d\|_2\|\bm a\|_2}{\sqrt m}.
\end{equation}
The right-hand side of \eqref{eq:coup:stab-remainder} has $L_t$ norm at most $C_tm^{-1/2}$ by \eqref{eq:coup:bmom} and H\"older's inequality. For the sample-dependent vector $\bm d$, the absolute difference between $N_S^{-1}\sum_{i\in\mathcal I_S}|\bm x_i^\top\bm d|\,|\eta_i|$ and $(2/\pi)\|\bm L_x\bm d\|_2$ is bounded by $\|\bm d\|_2$ times the supremum in \eqref{eq:coup:emplinear}.
Apply \eqref{eq:coup:emplinear} and \eqref{eq:coup:bmom} at order $2t$, and use H\"older's inequality. This bounds the difference between the empirical average multiplied by $2\sqrt{k/r}$ in \eqref{eq:coup:stab-remainder} and $\frac{4 \sqrt{k}}{\pi \sqrt{r}}\|\bm L_x\bm d\|_2$ by $C_tm^{-1/2}$ in $L_t$. Taking the maximum over $j\in S$ proves \eqref{eq:coup:stabapprox}.
\end{proof}

We now provide a formal proof of Lemma~\ref{lem:gaussian-agreement}:
\begin{proof}
We construct the Gaussian blocks $\bm Z_j$ from the finite sample and approximate the validation losses and empirical stability. These estimates then give the bound on $\mathbb P(\mathcal E_m^c)$ in \eqref{eq:decision-mismatch-rate}.

First, the common $\sigma^2$ terms cancel in $\Delta_m$, leaving a difference of quadratic forms. Their $q$th moments are bounded by \eqref{eq:coup:bmom}, with $t=2q$ and the Gaussian moments of $\bm{\beta}_\star$. Hence, $\|\Delta_m\|_{L_q}\leq C_q$. Moreover, the expression for $\Delta_\infty$ in \eqref{eq:finite-candidate-gap} is quadratic in the Gaussian vectors $\bm{B}$ and $\bm{\bar Z}$, so $\|\Delta_\infty\|_{L_q}\leq C_q$. Thus, if $d_m>1$, any joint construction satisfies the lemma after increasing $C_q$. We may therefore assume $d_m\leq1$, which implies $m>p$.

Write $\bm x_i=\bm L_x\bm g_i$ and $\bm\beta_\star=\bm L_\beta\bm B$, where $\bm L_\beta=\bm\Sigma_\beta^{1/2}$ and the $\bm g_i$ and $\bm B$ are independent standard Gaussian vectors. For fold $j$, let $\bm G_j$ contain the columns $\bm g_i$, let $\bm\eta_j$ contain the normalized noises $\eta_i$, and set
\begin{equation*}
 \bm W_j:=\frac1m\bm G_j\bm G_j^\top,
 \qquad
 \bm E_j:=\bm W_j^{-1/2}\frac{\bm G_j\bm\eta_j}{\sqrt m}.
\end{equation*}
Since $m>p$, $\bm W_j$ is invertible almost surely. Conditional on all designs, the $\bm E_j$ are independent $\mathcal N(\bm0,\bm I_p)$ vectors. Their conditional distribution is constant, so they are independent of the designs and of $\bm B$. We use these vectors in \eqref{eqn:zblockdef}. Substituting the definitions of $\bm W_j$ and $\bm E_j$ into \eqref{eq:fold-summaries} and \eqref{eqn:fold_u_defn} gives
\begin{align}\label{eq:coup:suffstat}
 \bm A_j&=\bm L_x\bm W_j\bm L_x,\notag\\
 \bm u_j&=\bm A_j\bm L_\beta\bm B
       +\sqrt{k/r}\,\bm L_x\bm W_j^{1/2}\bm E_j.
\end{align}
Moreover, for every $t \geq 1$, the Gaussian singular-value bound \cite[Corollary~5.35]{vershynin2012}, a union bound over the folds and integration of the tail give
\begin{equation*}
 \left\|\max_j\|\bm W_j-\bm I_p\|_{\mathrm{op}}\right\|_{L_t}
 \leq C_t\sqrt{\frac{p+\log(2k)}m},
\end{equation*}
for some $C_t < \infty$ depending only on $p, k, \kappa_1, \kappa_2, t$ and $R$. Also, $\Vert \bm{W}^{1/2}_j-\bm{I}_p\Vert_{\rm op}\leq \Vert \bm{W}_j-\bm{I}_p\Vert_{\rm op}$, since $\vert \sqrt{u}-1\vert \leq \vert u-1\vert$ for $u \geq 0$. Next, by \eqref{eq:diso}, $\Vert \bm{L}_x-\bm{I}_p\Vert_{\rm op}$ and $\Vert \bm{L}_{\beta}-\bm{I}_{p}\Vert_{\rm op}$ are at most $\delta_{\rm iso}$. Thus, substituting these bounds and the preceding estimates for $\bm{W}_j$ and $\bm{W}_j^{1/2}$ into \eqref{eq:coup:suffstat}, and using Gaussian moments, gives
\begin{equation}\label{eq:coup:basic}
 \left\|\max_j\bigl(\|\bm A_j-\bm I_p\|_{\mathrm{op}}
                   +\|\bm u_j-\bm Z_j\|_2\bigr)\right\|_{L_t}
 \leq C_td_m.
\end{equation}
For each nonempty $S$, we have
\begin{equation*}
 \widehat{\bm b}_{\ell,S}-c_\ell\bar{\bm Z}_S
 =(\bar{\bm A}_S+\kappa_\ell\bm I_p)^{-1}
 \left[\bar{\bm u}_S-\bar{\bm Z}_S
       +c_\ell(\bm I_p-\bar{\bm A}_S)\bar{\bm Z}_S\right].
\end{equation*}
The matrix $(\bar{\bm A}_S+\kappa_\ell\bm I_p)^{-1}$ has operator norm at most $\kappa_\ell^{-1}$.
For the product $(\bm I_p-\bar{\bm A}_S)\bar{\bm Z}_S$, apply H\"older's inequality with \eqref{eq:coup:basic} at order $2t$ and the Gaussian moments of $\bar{\bm Z}_S$.
This gives
\begin{equation}\label{eq:coup:fit-error}
 \max_{\ell,S\ne\varnothing}
 \|\widehat{\bm b}_{\ell,S}-c_\ell\bar{\bm Z}_S\|_{L_t}\leq C_td_m.
\end{equation}
To compare $\Delta_m$ and $\Delta_\infty$ in \eqref{eq:finite-candidate-gap}, set $\bm v_\ell=\widehat{\bm b}_{\ell,[k]}-\bm L_\beta\bm B$ and $\bm w_\ell=c_\ell\bar{\bm Z}-\bm B$. Then
\begin{equation*}
 |\bm v_\ell^\top\bm L_x^2\bm v_\ell-\|\bm w_\ell\|_2^2|
 \leq\|\bm L_x^2-\bm I_p\|_{\mathrm{op}}\|\bm v_\ell\|_2^2
 +(\|\bm v_\ell\|_2+\|\bm w_\ell\|_2)\|\bm v_\ell-\bm w_\ell\|_2.
\end{equation*}
By \eqref{eq:coup:bmom} and \eqref{eq:diso}, the expectation of the first term is at most a constant times $\delta_{\mathrm{iso}}$.
Equations \eqref{eq:coup:fit-error} and \eqref{eq:diso} bound $\|\bm v_\ell-\bm w_\ell\|_{L_2}$ by a constant times $d_m$.
Applying Cauchy--Schwarz to the second term and summing over $\ell\in\{1,2\}$ proves the first inequality in \eqref{eq:risk-coupling}.

It remains to bound $\mathbb P(\mathcal E_m^c)$. After dropping the common term $m^{-1}\sum_{i\in\mathcal N_j}y_i^2$, the validation loss on fold $j$ for a fit trained on $S$ is
\begin{equation}\label{eq:validation-functional}
 L_{j,\ell}^{(m)}(S)
 :=-2\bm u_j^\top\widehat{\bm b}_{\ell,S}
   +\widehat{\bm b}_{\ell,S}^\top\bm A_j\widehat{\bm b}_{\ell,S}.
\end{equation}
Let $L_{j,\ell}^{(\infty)}(S)$ be the same expression with $\bm A_j=\bm I_p$, $\bm u_j=\bm Z_j$ and $\widehat{\bm b}_{\ell,S}=c_\ell\bar{\bm Z}_S$. Subtracting the two versions of \eqref{eq:validation-functional} gives terms containing $\bm A_j-\bm I_p$, $\bm u_j-\bm Z_j$, or $\widehat{\bm b}_{\ell,S}-c_\ell\bar{\bm Z}_S$, each multiplied by at most two other factors. Apply \eqref{eq:coup:basic} and \eqref{eq:coup:fit-error} at order $3q$, and use \eqref{eq:coup:bmom} and Gaussian moments for the remaining factors. H\"older's inequality then gives
\begin{equation}\label{eq:coup:loss-error}
 \|L_{j,\ell}^{(m)}(S)-L_{j,\ell}^{(\infty)}(S)\|_{L_q}\leq C_qd_m.
\end{equation}
For $S\in\mathfrak S$, define
\begin{equation*}
 D_S(\bm Z):=\max_{j\in S}\|\bar{\bm Z}_S-\bar{\bm Z}_{S\setminus\{j\}}\|_2.
\end{equation*}
By \eqref{eq:diso} and \eqref{eq:coup:bmom}, replacing $\bm L_x$ by $\bm I_p$ in \eqref{eq:coup:stabapprox} adds at most $C_q\delta_{\mathrm{iso}}$ to the $L_q$ error. The reverse triangle inequality and \eqref{eq:coup:fit-error} then give
\begin{equation}\label{eq:coup:mucoup}
 \left\|\frac{\widehat\mu_{\ell,S}^{(m)}}{\sqrt m}
               -\frac{4\sqrt{k}}{\pi\sqrt{r}} c_\ell D_S(\bm Z)\right\|_{L_q}\leq C_qd_m.
\end{equation}

There are three kinds of comparisons to consider. The ordinary cross-validation difference on $S$ is
\begin{equation}\label{eq:ordinary-margin}
 Q_S^{(m)}:=\frac1{|S|}\sum_{j\in S}
 [L_{j,2}^{(m)}(S\setminus\{j\})-L_{j,1}^{(m)}(S\setminus\{j\})].
\end{equation}
For stability weight $\lambda_1$, the difference between the selection criteria at $\gamma_2$ and $\gamma_1$, divided by $\lambda_1\sqrt m$, is
\begin{equation}\label{eq:positive-margin}
 P_S^{(m)}:=\frac{Q_S^{(m)}}{\lambda_1\sqrt m}
              +\frac{\widehat\mu_{2,S}^{(m)}-\widehat\mu_{1,S}^{(m)}}{\sqrt m}.
\end{equation}
The Gaussian limit of $P_S^{(m)}$ is $P_S^{(\infty)}=\frac{4 \sqrt{k}}{\pi \sqrt{r}}(c_2-c_1)D_S(\bm Z)>0$ almost surely, which selects $\gamma_1$. Finally, suppose that the zero-weight inner rule chooses $\gamma_{\alpha_t}$ when fold $t$ is omitted, with $\bm\alpha\in\{1,2\}^k$, and the positive-weight inner rules choose $\gamma_1$. The positive-weight outer error minus the zero-weight outer error is
\begin{equation}\label{eq:outer-pattern-margin}
 O_{\bm\alpha}^{(m)}:=\frac1k\sum_{t=1}^k
 [L_{t,1}^{(m)}([k]\setminus\{t\})-L_{t,\alpha_t}^{(m)}([k]\setminus\{t\})].
\end{equation}
Define $Q_S^{(\infty)}$ and $O_{\bm\alpha}^{(\infty)}$ by replacing $L^{(m)}$ with $L^{(\infty)}$. 

Equation \eqref{eq:coup:loss-error} bounds the $L_q$ norms of $Q_S^{(m)}-Q_S^{(\infty)}$ and $O_{\bm\alpha}^{(m)}-O_{\bm\alpha}^{(\infty)}$ by $C_qd_m$.
It also gives $\|Q_S^{(m)}\|_{L_q}\leq C_q$, since $d_m\leq1$ and $Q_S^{(\infty)}$ is quadratic in Gaussian vectors.
Thus, the first term in \eqref{eq:positive-margin} has $L_q$ norm at most $C_qm^{-1/2}$.
Combining this estimate with \eqref{eq:coup:mucoup} gives $\|P_S^{(m)}-P_S^{(\infty)}\|_{L_q}\leq C_qd_m$.

Every $Q_S^{(\infty)}$ and every $O_{\bm\alpha}^{(\infty)}$ with at least one $\alpha_t=2$ is a nonzero quadratic polynomial in $\bm Z$: setting all blocks equal to a nonzero vector gives a negative value for $Q_S^{(\infty)}$ and a positive value for $O_{\bm\alpha}^{(\infty)}$. Their $L_2$ norms have a positive lower bound on $R$, since the Gaussian covariance is nonsingular and depends continuously on $r$. The Carbery--Wright inequality \cite[Theorem~8, with degree two and $q=4$ in their notation]{carbery2001distributional} thus gives
\begin{equation}\label{eq:coup:cw}
 \mathbb P\{|Q(\bm Z)|\leq\eta\}\leq C\sqrt\eta,
 \qquad 0<\eta\leq1,
\end{equation}
for each of these finitely many polynomials. For $P_S^{(\infty)}$, choose distinct $i,j\in S$. Since
$D_S(\bm Z)\geq|Z_{i,1}-Z_{j,1}|/[2(|S|-1)]$ and $Z_{i,1}-Z_{j,1}\sim\mathcal N(0,2k/r)$, the uniformly bounded density of $Z_{i,1}-Z_{j,1}$ gives $\mathbb P\{P_S^{(\infty)}\leq\eta\}\leq C\eta$, which also satisfies \eqref{eq:coup:cw}.

If every $Q_S^{(\infty)}$, $P_S^{(\infty)}$ and nonzero polynomial $O_{\bm\alpha}^{(\infty)}$ has absolute value greater than $\eta$ and differs from its finite-sample counterpart by at most $\eta$, the inner choices agree, as do the outer weight and final ridge parameter. The case $\alpha_t=1$ for every $t$ has $O_{\bm\alpha}^{(m)}=O_{\bm\alpha}^{(\infty)}=0$; we select the zero weight on equality in both constructions. For each $Q_S^{(\infty)}$, $P_S^{(\infty)}$ and nonzero $O_{\bm\alpha}^{(\infty)}$, \eqref{eq:coup:cw} bounds the probability of an absolute value at most $\eta$ by $C\sqrt\eta$. Markov's inequality and the preceding $L_q$ error bounds give probability at most $C_q(d_m/\eta)^q$ that a finite-sample comparison differs from its Gaussian counterpart by more than $\eta$.
A union bound therefore gives
\begin{equation}\label{eq:decision-mismatch}
 \mathbb P(\mathcal E_m^c)
 \leq C_q \max\{\sqrt{\eta}, (d_m/\eta)^q\}.
\end{equation}
Taking $\eta=d_m^{2q/(2q+1)}$ proves \eqref{eq:decision-mismatch-rate}.
\end{proof}

\subsection{Statement and Proof of Corollary \ref{cor:single}}\label{app:signs}

We first state Corollary \ref{cor:single}:
\begin{corollary}\label{cor:single}
Suppose that $\rho_m < 1$, that the transition band $B_m$ lies in the interior of $R$, and that $G_m$ is continuous on $R$. Then $G_m$ has at least one zero in $R$, every zero lies in
\begin{align*}
    B_m = \Bigl[ \frac{r^\star}{1+\rho_m}, \frac{r^\star}{1-\rho_m} \Bigr] ,
\end{align*}
and consequently every zero $r_m$ satisfies $|r_m - r^\star| \leq r^\star \rho_m/(1-\rho_m)$. If in addition $G_m$ is continuously differentiable on $R$ and
\begin{align*}
    \sup_{r \in B_m} \bigl| G_m'(r) - g'(r) \bigr|
    < \inf_{r \in B_m} g'(r) ,
\end{align*}
then the zero is unique. Finally, if the covariances are exactly isotropic and $r^\star$ lies in the interior of $R$, then $\rho_m = O\bigl(m^{-(q-1)/(2(2q+1))}\bigr)$, so provided $G_m$ is continuous on $R$, the hypotheses of the first part hold for all sufficiently large $m$ and every zero of $G_m$ satisfies $|r_m - r^\star| = O\bigl(m^{-(q-1)/(2(2q+1))}\bigr)$.
\end{corollary}

We now prove Corollary~\ref{cor:single}:
\begin{proof}
We first use Theorem~\ref{thm:phase} and continuity to prove the existence of a zero and locate every zero within $B_m$. We then use the derivative condition for uniqueness and the isotropic error bound for the rate at which the zeros approach $r^\star$.

Since $\rho_m < 1$ and $r^\star \in R$, the band $B_m$ contains $r^\star$ and is nonempty. If $\inf R \geq r^\star/(1+\rho_m)$, then $\inf R \in B_m$, contradicting $B_m \subseteq \operatorname{int}(R)$. Thus $\inf R < r^\star/(1+\rho_m)$. The same argument at $\sup R$ gives $\sup R > r^\star/(1-\rho_m)$. Consequently, $B_m$ is the full interval stated in Corollary~\ref{cor:single}, and $R$ contains points on either side of $B_m$. Theorem \ref{thm:phase} gives $G_m(\inf R)<0<G_m(\sup R)$. Continuity therefore gives at least one zero in $R$, while Theorem \ref{thm:phase} places every zero in $B_m$. The error bound follows because $r^\star/(1-\rho_m) - r^\star = r^\star\rho_m/(1-\rho_m)$ exceeds $r^\star - r^\star/(1+\rho_m) = r^\star\rho_m/(1+\rho_m)$.

Under the additional hypotheses, for every $r \in B_m$,
\begin{align*}
    G_m'(r) \;\geq\; \inf_{u \in B_m} g'(u)
      - \sup_{u \in B_m} \bigl| G_m'(u) - g'(u) \bigr| \;>\; 0 ,
\end{align*}
so $G_m$ is strictly increasing on $B_m$ and its zero there is unique; since every zero lies in $B_m$, it is unique in $R$.

Finally, suppose the covariances are exactly isotropic, so $\delta_{\mathrm{iso}} = 0$. Then \eqref{eq:approx} reduces to
\begin{align*}
    \varepsilon_m \;\leq\; C_q\Bigl(\tfrac{p+\log(2k)}m\Bigr)^{\frac{q-1}{2(2q+1)}} ,
\end{align*}
so $\varepsilon_m = O(m^{-(q-1)/(2(2q+1))})$ with $p$, $k$, $\gamma_1, \gamma_2$ and $\lambda_1$ fixed. As $s > 0$ and $\Pi_R > 0$ by Proposition \ref{prop:ref}, the same order holds for $\rho_m = \varepsilon_m/(s\Pi_R)$; in particular $\rho_m < 1$ for large $m$, and $B_m$ shrinks to $\{r^\star\}$, so $B_m \subseteq \operatorname{int}(R)$ for large $m$ because $r^\star$ is interior to $R$. Under the stated continuity assumption, the first part therefore applies, and the error bound gives $|r_m - r^\star| = O(\rho_m)$.
\end{proof}

\subsection{Proof of Proposition \ref{prop:magnitude}}\label{app:prop8proof}

\begin{proof}
Write $c_\ell:=(1+\kappa_\ell)^{-1}$ for $\ell\in\{1,2\}$, and let $d:=c_1+c_2\in(0,2)$.
Let $V$ denote the sum of squared deviations of all $k$ fold vectors from their mean:
\begin{align*}
 V:=\sum_{j=1}^k\|\bm Z_j-\bar{\bm Z}\|_2^2.
\end{align*}
After deleting fold $t$, the corresponding sum is
\begin{align*}
 V^{(t)}:=\sum_{j\ne t}\|\bm Z_j-\bar{\bm Z}_{-t}\|_2^2.
\end{align*}
For a comparison using $K$ folds, define the cutoff for squared fold dispersion relative to the squared norm of the mean by
\begin{align*}
 t_K:=\frac{(2-d)K(K-1)^2}{2(K-1)+d}\qquad(K\geq2).
\end{align*}
Since $\bar{\bm Z}_{-j}=(k\bar{\bm Z}-\bm Z_j)/(k-1)$, direct expansion gives
\begin{align*}
 \frac1k\sum_{j=1}^k\langle\bm Z_j,\bar{\bm Z}_{-j}\rangle
 &=\|\bar{\bm Z}\|_2^2-\frac{V}{k(k-1)},\\
 \frac1k\sum_{j=1}^k\|\bar{\bm Z}_{-j}\|_2^2
 &=\|\bar{\bm Z}\|_2^2+\frac{V}{k(k-1)^2}.
\end{align*}
Substituting these identities into the cross-validation score \eqref{eq:HK} yields
\begin{equation}\label{eq:mag-comparison}
 H_k(\kappa_2,\bm Z)-H_k(\kappa_1,\bm Z)
 =(c_1-c_2)\left[(2-d)\|\bar{\bm Z}\|_2^2
 -\frac{2(k-1)+d}{k(k-1)^2}V\right].
\end{equation}
As $c_1<c_2$, full-sample cross-validation selects $\gamma_2$ exactly when $V<t_k\|\bar{\bm Z}\|_2^2$, apart from ties. Applying \eqref{eq:mag-comparison} to the $k-1$ folds other than $t$ shows that the inner comparison selects $\gamma_2$ on
\begin{align*}
 A_t:=\{V^{(t)}<t_{k-1}\|\bar{\bm Z}_{-t}\|_2^2\}.
\end{align*}
The inner ties have probability zero because the Gaussian mean and centered deviations are independent, so the ratio $V^{(t)}/\|\bar{\bm Z}_{-t}\|_2^2$ has a continuous distribution. By \eqref{eq:reference-outer},
\begin{equation}\label{eq:mag-outer}
 M_1(\bm Z)-M_0(\bm Z)=\frac{c_2-c_1}{k}\sum_{t=1}^k1_{A_t}
 \langle\bar{\bm Z}_{-t},2\bm Z_t-d\bar{\bm Z}_{-t}\rangle.
\end{equation}
If every $A_t$ occurs, the sum in \eqref{eq:mag-outer} equals $k$ times the bracket in \eqref{eq:mag-comparison}. It is therefore positive whenever full-sample cross-validation selects $\gamma_2$. Since $D_{\mathrm{ref}}$ also requires $M_1<M_0$, we obtain
\begin{equation}\label{eq:mag-containment}
 D_{\mathrm{ref}}\subseteq\bigcup_{t=1}^k A_t^c
 \qquad\text{up to a null set}.
\end{equation}
These comparisons also show that $D_{\mathrm{ref}}$ depends on the candidate pair only through $d$ and is unchanged by multiplying every fold vector by the same positive scalar.

To prove \eqref{eq:pi-rate}, let $\bar{\bm E}$ be the average of the standard Gaussian fold vectors:
\begin{align*}
 \bar{\bm E}:=\frac1k\sum_{j=1}^k\bm E_j.
\end{align*}
Denote their centered deviations and the array of these deviations by
\begin{align*}
 \bm D_j&:=\bm E_j-\bar{\bm E},\\
 \bm D&:=(\bm D_1,\ldots,\bm D_k).
\end{align*}
Rescaling the fold vectors in \eqref{eqn:zblockdef} by $\sqrt{r/k}$ gives $\widetilde{\bar{\bm Z}}+\bm D_j$, whose average is
\begin{align*}
 \widetilde{\bar{\bm Z}}&:=\sqrt{r/k}\,\bar{\bm Z}\\
 &=\sqrt{r/k}\,\bm B+\bar{\bm E}.
\end{align*}
This rescaled average is independent of $\bm D$ and has distribution
\begin{align*}
 \widetilde{\bar{\bm Z}}\sim\mathcal N\left(\bm0,\frac{1+r}{k}\bm I_p\right).
\end{align*}
For a deterministic $\bm x\in\mathbb R^p$, write $\bm x+\bm D$ for the fold array $(\bm x+\bm D_1,\ldots,\bm x+\bm D_k)$. Scale invariance of $D_{\mathrm{ref}}$ and the identity $\bar{\bm Z}=\sqrt{k/r}\,\widetilde{\bar{\bm Z}}$ allow us to rewrite \eqref{eq:Pi} as
\begin{align*}
 \Pi(r)=\frac{k}{p(1+r)}
 \mathbb E\left[
 \|\widetilde{\bar{\bm Z}}\|_2^2
 1_{D_{\mathrm{ref}}}(\widetilde{\bar{\bm Z}}+\bm D)
 \right].
\end{align*}
The rescaled mean is independent of $\bm D$.
Substituting its Gaussian density and multiplying by $(1+r)^\alpha$, with $\alpha:=p/2+1$, gives
\begin{align}\label{eq:mag-rescaled}
 (1+r)^\alpha\Pi(r)
 =\frac{k^\alpha}{p(2\pi)^{p/2}}
 \mathbb E_{\bm D}\int_{\mathbb R^p}
 \|\bm x\|_2^2 1_{D_{\mathrm{ref}}}(\bm x+\bm D)
 \exp\left(-\frac{k\|\bm x\|_2^2}{2(1+r)}\right)\,d\bm x.
\end{align}
The integrand increases with $r$. Setting $r=0$ on the right-hand side therefore gives the lower bound with
\begin{equation}\label{eq:mag-lower-constant}
 c_\Pi:=\frac{k}{p}\mathbb E\left[\|\bar{\bm E}\|_2^2
 1_{D_{\mathrm{ref}}}(\bm E_1,\ldots,\bm E_k)\right]>0.
\end{equation}
Strict positivity follows from the open subset of $D_{\mathrm{ref}}$ with nonzero mean constructed in Appendix~\ref{app:reference}.

For the upper bound, let $W$ denote the total squared size of the centered Gaussian fold vectors:
\begin{align*}
 W:=\sum_{j=1}^k\|\bm D_j\|_2^2.
\end{align*}
Define the constant that will bound the size of the fold average by
\begin{align*}
 L:=(k-1)^{-1}+t_{k-1}^{-1/2}.
\end{align*}

If $\bm x+\bm D\in D_{\mathrm{ref}}$, then \eqref{eq:mag-containment} gives an index $t$ for which $A_t$ fails. The mean of the remaining fold vectors is $\bm x-\bm D_t/(k-1)$. Their sum of squared deviations is at most $W$, so the definition of $A_t$ gives
\begin{align*}
\left\|\bm x-\frac{\bm D_t}{k-1}\right\|_2 \leq \sqrt{\frac{W}{t_{k-1}}}.
\end{align*}
The triangle inequality therefore gives
 \begin{align*}
\|\bm x\|_2
 &\leq\frac{\|\bm D_t\|_2}{k-1}+\sqrt{W/t_{k-1}}\\
 &\leq L\sqrt W.
\end{align*}
Dropping the exponential in \eqref{eq:mag-rescaled} and integrating over this ball gives
\begin{equation}\label{eq:mag-radial-bound}
 (1+r)^\alpha\Pi(r)\leq\frac{k^\alpha L^{p+2}}{2^\alpha\Gamma(p/2+2)}\mathbb E[W^\alpha].
\end{equation}
Since $W\sim\chi^2_{(k-1)p}$, its moment at order $\alpha=p/2+1$ is
\begin{equation}\label{eq:mag-chi-moment}
\begin{aligned}
 \mathbb E[W^\alpha]
 &=2^\alpha\frac{\Gamma((k-1)p/2+\alpha)}{\Gamma((k-1)p/2)}\\
 &=2^\alpha\frac{\Gamma(kp/2+1)}{\Gamma((k-1)p/2)}.
\end{aligned}
\end{equation}
Substituting \eqref{eq:mag-chi-moment} and the definitions of $L$ and $t_{k-1}$ into \eqref{eq:mag-radial-bound} gives the upper bound in \eqref{eq:pi-rate} with
\begin{equation}\label{eq:pi-upper-constant}
 C_\Pi:=
 \frac{k^{p/2+1}\Gamma(kp/2+1)}
 {\Gamma(p/2+2)\Gamma((k-1)p/2)}
 \left(
 \frac{1}{k-1}
 +\sqrt{\frac{2(k-2)+d}
 {(2-d)(k-1)(k-2)^2}}
 \right)^{p+2},
\end{equation}
where $\Gamma$ denotes the gamma function. Together with \eqref{eq:mag-lower-constant}, this proves \eqref{eq:pi-rate}.
\end{proof}

\section{Supplementary Details on Real Data Experiments}

\subsection{Dataset Description}\label{ssec:dataset}
We use a variety of real datasets from the literature in our computational experiments. The information of each dataset is summarized in Table~\ref{tab:datasets}. Note that we increased the number of features on selected datasets by including second-order interactions. 

\begin{table}[h]
\centering\footnotesize
\begin{tabular}{@{}l r r c c@{}} \toprule
Dataset & n & p & Notes & Reference\\
\midrule

Housing & 506 & 13& &\cite{gomez2021mixed}\\ 
Wine & 6497 & 11&&\cite{gomez2021mixed} \\ 
AutoMPG & 392 & 25&&\cite{gomez2021mixed}\\ 
\multirow{2}{*}{Hitters} & \multirow{2}{*}{263} & \multirow{2}{*}{19}&Removed rows with missing data& \multirow{2}{*}{Kaggle}\\ 
&&&$\bm{y}=\log(\text{salary})$&\\
Prostate & 97 & 8&&R Package \texttt{ncvreg} \\ 
Servo & 167 & 19&One-hot encoding of features &\cite{gomez2021mixed}  \\ 
Toxicity & 38 & 9&&\cite{rousseeuw2009robustbase} \\ 
Steam & 25 & 8&&\cite{rousseeuw2009robustbase} \\ 
Alcohol2 & 44 & 21&2nd order interactions added&\cite{rousseeuw2009robustbase} \\ 
\midrule
TopGear & 242 & 373&&\cite{bottmer2022sparse}\\ 
Bardet & 120 & 200 &&\cite{ye2018sparsity}\\ 
Vessel & 180 & 486&&\cite{christidis2020split} \\ 
Riboflavin & 71 & 4088&&R package \texttt{hdi}\\
\bottomrule
\end{tabular}
\caption{Real datasets used.
}
\label{tab:datasets}
\end{table}

Our sources for these datasets are as follows:
\begin{itemize}
    \item Four UCI datasets: AutoMPG, Housing, Servo, and Wine. We obtained these datasets from the online supplement to \cite{gomez2021mixed}. 
    \item The \verb|alcohol| dataset distributed via the \verb|R| package \verb|robustbase|. Note that we increased the number of features for this dataset by including second-order interactions.
    \item The \verb|Bardet| dataset provided by \cite{ye2018sparsity}.
    \item The \verb|hitters| Kaggle dataset, after preprocessing the dataset to remove rows with any missing entries, and transforming the response by taking log(Salary), as is standard when predicting salaries via regression.
    \item The \verb|Prostate| dataset 
    distributed via the R package \verb|ncvreg|.
    \item The \verb|Riboflavin| dataset distributed by the \verb|R| package \verb|hdi|.
    \item The \verb|steamUse| dataset provided by \cite{rousseeuw2009robustbase}.
    \item The \verb|topgear| dataset provided by \cite{bottmer2022sparse}.
    \item The \verb|toxicity| dataset provided by \cite{rousseeuw2009robustbase}.
    \item The \verb|vessel| dataset made publicly available by \cite{christidis2020split}.
\end{itemize}

\subsection{Implementation Details for Sparse Ridge Regression}\label{ssec:sup_ridge_implementation}
For computational tractability, we solve the sparse-ridge problem using the perspective
reformulation \citep{gunluk2010perspective, atamturk2019rank}
\begin{align}
\min_{\bm{\beta}\in\mathbb{R}^p, \bm{z}\in\{0,1\}^p}
\frac{1}{2\gamma}\sum_{j=1}^p \frac{\beta_j^2}{z_j}
+\rho \Vert \bm{\beta}\Vert_1+
\frac{1}{2}\|\bm{X\beta}-\bm{y}\|_2^2
\quad
\mathrm{s.t.}\quad
\sum_{j=1}^p z_j\le \tau,
\end{align}
with the usual convention that the perspective terms impose $\beta_j=0$ whenever
$z_j=0$. We obtain approximate solutions using the greedy rounding algorithm of
\cite{ bertsimas2021unified}.  

We perform cyclic coordinate descent over the three hyperparameters $(\tau,\rho, \gamma)$ with a limit of $10$ iterations, and use the same procedure inside the nested stability-regularized procedure. When refitting on the full training/validation set after cross-validation, we use the selected hyperparameters $(\tau^\star,\frac{n}{n_{\mathrm{train}}}\rho^\star, n_{\mathrm{train}}\gamma^\star/n)$, where $n_{\mathrm{train}}$ is the number of observations used in each fold-deleted training set. This rescaling ensures that both regularization coefficients are proportional to the dataset size.

For the cyclic coordinate descent schemes, we restrict $\tau$ to satisfy $\tau\log\tau\le \min(n,p)$. This restriction is motivated by the sparsity condition $\tau\log\tau=O(n)$ used in the analysis of \citet{gamarnik2022sparse}. We use 21 values of $\gamma$ log-uniformly distributed over $[10^{-3},10^3]$, ten values of $\rho$, namely $0$ plus nine values log-uniformly distributed on $[10^{-4}\rho_{\max}, \rho_{\max}]$, as in \citet[Section 2.5]{friedman2010regularization}—where $\rho_{\max}=\Vert\bm{X}^\top \bm{y}\Vert_\infty$ is computed separately for each train/test split using the standardized training predictors and centered training response—and the grid of candidate $\lambda$ values $\{0, 10^{-2}, 10^{-1}, 1, 10\}$.

Finally, when implementing \citet{breiman2017classification}'s one-standard-error rule as a benchmark comparison, we use the following procedure: for foldwise validation errors $e_j(\theta)$ and fold sizes $n_j: \sum_{j \in [k]}n_j=n$, let $w_j=n_j/n$ and compute
\begin{align*}
\mathrm{CV}(\theta)
&=\sum_{j=1}^{5}w_j e_j(\theta),\\
\widehat{\mathrm{SE}}(\theta)
&=\left[\frac{1}{4}\sum_{j=1}^{5}w_j
\bigl(e_j(\theta)-\mathrm{CV}(\theta)\bigr)^2\right]^{1/2}.
\end{align*}
Writing $\widehat\theta_{\mathrm{CV}}$ for the candidate selected by ordinary kCV, we choose the lexicographically smallest $(\tau,-\rho,\gamma)$ among evaluated triples satisfying
\begin{equation*}
\mathrm{CV}(\theta)\leq
\mathrm{CV}(\widehat\theta_{\mathrm{CV}})
+\widehat{\mathrm{SE}}(\widehat\theta_{\mathrm{CV}}).
\end{equation*}
This defines parsimony first by the support-size limit, then by stronger $\ell_1$ and ridge penalties.


{\footnotesize
\begin{thebibliography}{67}
    \providecommand{\natexlab}[1]{#1}
    \providecommand{\url}[1]{\texttt{#1}}
    \providecommand{\urlprefix}{URL }
    
    \bibitem[{Arlot \protect\BIBand{} Celisse(2010)}]{arlot2010survey}
    Arlot S, Celisse A (2010) A survey of cross-validation procedures for model selection. \emph{Statistics Surveys} 4:40--79.
    
    \bibitem[{Atamt{\"u}rk \protect\BIBand{} G{\'o}mez(2025)}]{atamturk2019rank}
    Atamt{\"u}rk A, G{\'o}mez A (2025) Rank-one convexification for sparse regression. \emph{Journal of Machine Learning Research} 26(35):1--50.
    
    \bibitem[{Austern \protect\BIBand{} Zhou(2025)}]{austern2025asymptotics}
    Austern M, Zhou W (2025) Asymptotics of cross-validation. \emph{Annales de l'Institut Henri Poincar\'e, Probabilit\'es et Statistiques} 61:2804--2865.
    
    \bibitem[{Bates et~al.(2024)Bates, Hastie, \protect\BIBand{} Tibshirani}]{bates2021cross}
    Bates S, Hastie T, Tibshirani R (2024) Cross-validation: {W}hat does it estimate and how well does it do it? \emph{Journal of the American Statistical Association} 119:1434--1445.
    
    \bibitem[{Bayle et~al.(2020)Bayle, Bayle, Janson, \protect\BIBand{} Mackey}]{bayle2020cross}
    Bayle P, Bayle A, Janson L, Mackey L (2020) Cross-validation confidence intervals for test error. \emph{Advances in Neural Information Processing Systems} 33:16339--16350.
    
    \bibitem[{Bengio \protect\BIBand{} Grandvalet(2004)}]{bengio2004no}
    Bengio Y, Grandvalet Y (2004) No unbiased estimator of the variance of {$k$}-fold cross-validation. \emph{Journal of Machine Learning Research} 5:1089--1105.
    
    \bibitem[{Bertsimas \protect\BIBand{} Copenhaver(2018)}]{bertsimas2018characterization}
    Bertsimas D, Copenhaver MS (2018) Characterization of the equivalence of robustification and regularization in linear and matrix regression. \emph{European Journal of Operational Research} 270(3):931--942.
    
    \bibitem[{Bertsimas et~al.(2021)Bertsimas, Cory-Wright, \protect\BIBand{} Pauphilet}]{bertsimas2021unified}
    Bertsimas D, Cory-Wright R, Pauphilet J (2021) A unified approach to mixed-integer optimization problems with logical constraints. \emph{SIAM Journal on Optimization} 31(3):2340--2367.
    
    \bibitem[{Bishop(2006)}]{bishop2006}
    Bishop CM (2006) \emph{Pattern Recognition and Machine Learning} (Springer).
    
    \bibitem[{Blum \protect\BIBand{} Hardt(2015)}]{blum2015ladder}
    Blum A, Hardt M (2015) The ladder: A reliable leaderboard for machine learning competitions. \emph{International Conference on Machine Learning}, 1006--1014 (PMLR).
    
    \bibitem[{Bottmer et~al.(2022)Bottmer, Croux, \protect\BIBand{} Wilms}]{bottmer2022sparse}
    Bottmer L, Croux C, Wilms I (2022) Sparse regression for large data sets with outliers. \emph{European Journal of Operational Research} 297(2):782--794.
    
    \bibitem[{Boucheron et~al.(2013)Boucheron, Lugosi, \protect\BIBand{} Massart}]{boucheron2013concentration}
    Boucheron S, Lugosi G, Massart P (2013) \emph{Concentration Inequalities: A Nonasymptotic Theory of Independence} (Oxford: Oxford University Press).
    
    \bibitem[{Bousquet \protect\BIBand{} Elisseeff(2002)}]{bousquet2002stability}
    Bousquet O, Elisseeff A (2002) Stability and generalization. \emph{Journal of Machine Learning Research} 2:499--526.
    
    \bibitem[{Breiman(1996)}]{breiman1996heuristics}
    Breiman L (1996) Heuristics of instability and stabilization in model selection. \emph{The Annals of Statistics} 24(6):2350--2383.
    
    \bibitem[{Breiman et~al.(1984)Breiman, Friedman, Olshen, \protect\BIBand{} Stone}]{breiman2017classification}
    Breiman L, Friedman JH, Olshen RA, Stone CJ (1984) \emph{Classification and Regression Trees} (Wadsworth \& Brooks/Cole).
    
    \bibitem[{Carbery \protect\BIBand{} Wright(2001)}]{carbery2001distributional}
    Carbery A, Wright J (2001) Distributional and {$L^q$} norm inequalities for polynomials over convex bodies in {$\mathbb R^n$}. \emph{Mathematical Research Letters} 8(3):233--248.
    
    \bibitem[{Cawley \protect\BIBand{} Talbot(2010)}]{cawley2010over}
    Cawley GC, Talbot NL (2010) On over-fitting in model selection and subsequent selection bias in performance evaluation. \emph{Journal of Machine Learning Research} 11:2079--2107.
    
    \bibitem[{Chen \protect\BIBand{} Yang(2021)}]{chen2021one}
    Chen Y, Yang Y (2021) The one standard error rule for model selection: {D}oes it work? \emph{Stats} 4(4):868--892.
    
    \bibitem[{Christidis et~al.(2020)Christidis, Lakshmanan, Smucler, \protect\BIBand{} Zamar}]{christidis2020split}
    Christidis AA, Lakshmanan L, Smucler E, Zamar R (2020) Split regularized regression. \emph{Technometrics} 62(3):330--338.
    
    \bibitem[{Cover \protect\BIBand{} Thomas(2006)}]{cover2006elements}
    Cover TM, Thomas JA (2006) \emph{Elements of Information Theory} (Hoboken, NJ: Wiley-Interscience), 2nd edition.
    
    \bibitem[{Devroye \protect\BIBand{} Wagner(1979)}]{devroye1979distribution}
    Devroye L, Wagner TJ (1979) Distribution-free performance bounds for potential function rules. \emph{IEEE Transactions on Information Theory} 25(5):601--604.
    
    \bibitem[{Friedman et~al.(2010)Friedman, Hastie, \protect\BIBand{} Tibshirani}]{friedman2010regularization}
    Friedman JH, Hastie T, Tibshirani R (2010) Regularization paths for generalized linear models via coordinate descent. \emph{Journal of Statistical Software} 33:1--22.
    
    \bibitem[{Gamarnik \protect\BIBand{} Zadik(2022)}]{gamarnik2022sparse}
    Gamarnik D, Zadik I (2022) Sparse high-dimensional linear regression. {E}stimating squared error and a phase transition. \emph{The Annals of Statistics} 50(2):880--903.
    
    \bibitem[{Gao et~al.(2024)Gao, Chen, \protect\BIBand{} Kleywegt}]{gao2024wasserstein}
    Gao R, Chen X, Kleywegt AJ (2024) Wasserstein distributionally robust optimization and variation regularization. \emph{Operations Research} 72(3):1177--1191.
    
    \bibitem[{Geisser(1975)}]{geisser1975predictive}
    Geisser S (1975) The predictive sample reuse method with applications. \emph{Journal of the American Statistical Association} 70(350):320--328.
    
    \bibitem[{G{\'o}mez \protect\BIBand{} Prokopyev(2021)}]{gomez2021mixed}
    G{\'o}mez A, Prokopyev OA (2021) A mixed-integer fractional optimization approach to best subset selection. \emph{INFORMS Journal on Computing} 33(2):551--565.
    
    \bibitem[{G{\"u}nl{\"u}k \protect\BIBand{} Linderoth(2010)}]{gunluk2010perspective}
    G{\"u}nl{\"u}k O, Linderoth J (2010) Perspective reformulations of mixed integer nonlinear programs with indicator variables. \emph{Mathematical Programming} 124(1--2):183--205.
    
    \bibitem[{Gupta et~al.(2024)Gupta, Huang, \protect\BIBand{} Rusmevichientong}]{gupta2024debiasing}
    Gupta V, Huang M, Rusmevichientong P (2024) Debiasing in-sample policy performance for small-data, large-scale optimization. \emph{Operations Research} 72(2):848--870.
    
    \bibitem[{Gupta \protect\BIBand{} Rusmevichientong(2021)}]{gupta2021small}
    Gupta V, Rusmevichientong P (2021) Small-data, large-scale linear optimization with uncertain objectives. \emph{Management Science} 67(1):220--241.
    
    \bibitem[{Harrison \protect\BIBand{} March(1984)}]{harrison1984decision}
    Harrison JR, March JG (1984) Decision making and postdecision surprises. \emph{Administrative Science Quarterly} 29(1):26--42.
    
    \bibitem[{Hastie et~al.(2009)Hastie, Tibshirani, \protect\BIBand{} Friedman}]{hastie2009elements}
    Hastie T, Tibshirani R, Friedman JH (2009) \emph{The Elements of Statistical Learning: Data Mining, Inference, and Prediction} (Springer), 2nd edition.
    
    \bibitem[{Hoerl \protect\BIBand{} Kennard(1970)}]{hoerl1970ridge}
    Hoerl AE, Kennard RW (1970) Ridge regression: Biased estimation for nonorthogonal problems. \emph{Technometrics} 12(1):55--67.
    
    \bibitem[{Homrighausen \protect\BIBand{} McDonald(2013)}]{homrighausen2013lasso}
    Homrighausen D, McDonald D (2013) The lasso, persistence, and cross-validation. \emph{International Conference on Machine Learning}, 1031--1039 (PMLR).
    
    \bibitem[{Kale et~al.(2011)Kale, Kumar, \protect\BIBand{} Vassilvitskii}]{kale2011cvstability}
    Kale S, Kumar R, Vassilvitskii S (2011) Cross-validation and mean-square stability. \emph{Proceedings of the 2nd Symposium on Innovations in Computer Science (ICS)}, 487--495.
    
    \bibitem[{Kearns \protect\BIBand{} Ron(1997)}]{kearns1997algorithmic}
    Kearns M, Ron D (1997) Algorithmic stability and sanity-check bounds for leave-one-out cross-validation. \emph{Proceedings of the Tenth Annual Conference on Computational Learning Theory}, 152--162 (ACM).
    
    \bibitem[{Kim(2009)}]{kim2009estimating}
    Kim JH (2009) Estimating classification error rate: Repeated cross-validation, repeated hold-out and bootstrap. \emph{Computational Statistics \& Data Analysis} 53(11):3735--3745.
    
    \bibitem[{Krstajic et~al.(2014)Krstajic, Buturovic, Leahy, \protect\BIBand{} Thomas}]{krstajic2014cross}
    Krstajic D, Buturovic LJ, Leahy DE, Thomas S (2014) Cross-validation pitfalls when selecting and assessing regression and classification models. \emph{Journal of Cheminformatics} 6(1):10.
    
    \bibitem[{Larson(1931)}]{larson1931shrinkage}
    Larson SC (1931) The shrinkage of the coefficient of multiple correlation. \emph{Journal of Educational Psychology} 22(1):45--55.
    
    \bibitem[{Ledoux \protect\BIBand{} Talagrand(1991)}]{ledoux1991}
    Ledoux M, Talagrand M (1991) \emph{Probability in {Banach} Spaces: Isoperimetry and Processes}, volume~23 of \emph{Ergebnisse der Mathematik und ihrer Grenzgebiete, 3. Folge} (Berlin: Springer-Verlag).
    
    \bibitem[{Lei(2020)}]{lei2020cross}
    Lei J (2020) Cross-validation with confidence. \emph{Journal of the American Statistical Association} 115(532):1978--1997.
    
    \bibitem[{Lei(2025)}]{lei2025modern}
    Lei J (2025) A modern theory of cross-validation through the lens of stability. \emph{Foundations and Trends in Statistics} 1(3-4):391--548.
    
    \bibitem[{Lim \protect\BIBand{} Yu(2016)}]{limyu2016escv}
    Lim C, Yu B (2016) Estimation stability with cross-validation ({ESCV}). \emph{Journal of Computational and Graphical Statistics} 25(2):464--492.
    
    \bibitem[{Liu et~al.(2010)Liu, Roeder, \protect\BIBand{} Wasserman}]{liu2010stability}
    Liu H, Roeder K, Wasserman L (2010) Stability approach to regularization selection ({StARS}) for high dimensional graphical models. \emph{Advances in Neural Information Processing Systems} 23:1432--1440.
    
    \bibitem[{Meinshausen \protect\BIBand{} B{\"u}hlmann(2010)}]{meinshausen2010stability}
    Meinshausen N, B{\"u}hlmann P (2010) Stability selection. \emph{Journal of the Royal Statistical Society Series B: Statistical Methodology} 72(4):417--473.
    
    \bibitem[{Michaud(1989)}]{michaud1989markowitz}
    Michaud RO (1989) The {M}arkowitz optimization enigma: {I}s ‘optimized’ optimal? \emph{Financial Analysts Journal} 45(1):31--42.
    
    \bibitem[{Mosier(1951)}]{mosier1951problems}
    Mosier CI (1951) Problems and designs of cross-validation. \emph{Educational and Psychological Measurement} 11(1):5--11.
    
    \bibitem[{Mukherjee et~al.(2006)Mukherjee, Niyogi, Poggio, \protect\BIBand{} Rifkin}]{mukherjee2006learning}
    Mukherjee S, Niyogi P, Poggio T, Rifkin R (2006) Learning theory: {S}tability is sufficient for generalization and necessary and sufficient for consistency of empirical risk minimization. \emph{Advances in Computational Mathematics} 25(1--3):161--193.
    
    \bibitem[{Ng(1997)}]{ng1997preventing}
    Ng AY (1997) Preventing ``overfitting'' of cross-validation data. \emph{International Conference on Machine Learning}, 245--253.
    
    \bibitem[{Patil et~al.(2021)Patil, Wei, Rinaldo, \protect\BIBand{} Tibshirani}]{patil2021uniform}
    Patil P, Wei Y, Rinaldo A, Tibshirani R (2021) Uniform consistency of cross-validation estimators for high-dimensional ridge regression. \emph{International Conference on Artificial Intelligence and Statistics}, 3178--3186 (PMLR).
    
    \bibitem[{Rao et~al.(2008)Rao, Fung, \protect\BIBand{} Rosales}]{rao2008dangers}
    Rao RB, Fung G, Rosales R (2008) On the dangers of cross-validation. {A}n experimental evaluation. \emph{Proceedings of the 2008 SIAM International Conference on Data Mining}, 588--596 (SIAM).
    
    \bibitem[{Reunanen(2003)}]{reunanen2003overfitting}
    Reunanen J (2003) Overfitting in making comparisons between variable selection methods. \emph{Journal of Machine Learning Research} 3:1371--1382.
    
    \bibitem[{Roelofs et~al.(2019)Roelofs, Shankar, Recht, Fridovich-Keil, Hardt, Miller, \protect\BIBand{} Schmidt}]{roelofs2019meta}
    Roelofs R, Shankar V, Recht B, Fridovich-Keil S, Hardt M, Miller J, Schmidt L (2019) A meta-analysis of overfitting in machine learning. \emph{Advances in Neural Information Processing Systems} 32:9179--9189.
    
    \bibitem[{Rogers \protect\BIBand{} Wagner(1978)}]{rogers1978finite}
    Rogers WH, Wagner TJ (1978) A finite sample distribution-free performance bound for local discrimination rules. \emph{The Annals of Statistics} 6(3):506--514.
    
    \bibitem[{Rousseeuw et~al.(2009)Rousseeuw, Croux, Todorov, Ruckstuhl, Salibian-Barrera, Verbeke, Koller, \protect\BIBand{} Maechler}]{rousseeuw2009robustbase}
    Rousseeuw P, Croux C, Todorov V, Ruckstuhl A, Salibian-Barrera M, Verbeke T, Koller M, Maechler M (2009) \texttt{robustbase}: {B}asic robust statistics. \emph{R package version 0.4-5} .
    
    \bibitem[{Shalev-Shwartz et~al.(2010)Shalev-Shwartz, Shamir, Srebro, \protect\BIBand{} Sridharan}]{shalev2010learnability}
    Shalev-Shwartz S, Shamir O, Srebro N, Sridharan K (2010) Learnability, stability and uniform convergence. \emph{Journal of Machine Learning Research} 11:2635--2670.
    
    \bibitem[{Shao(1993)}]{shao1993linear}
    Shao J (1993) Linear model selection by cross-validation. \emph{Journal of the American Statistical Association} 88(422):486--494.
    
    \bibitem[{Smith \protect\BIBand{} Winkler(2006)}]{smith2006optimizer}
    Smith JE, Winkler RL (2006) The optimizer’s curse: Skepticism and postdecision surprise in decision analysis. \emph{Management Science} 52(3):311--322.
    
    \bibitem[{Stephenson et~al.(2021)Stephenson, Frangella, Udell, \protect\BIBand{} Broderick}]{stephenson2021can}
    Stephenson W, Frangella Z, Udell M, Broderick T (2021) Can we globally optimize cross-validation loss? {Q}uasiconvexity in ridge regression. \emph{Advances in Neural Information Processing Systems} 34:24352--24364.
    
    \bibitem[{Stone(1974)}]{stone1974cross}
    Stone M (1974) Cross-validatory choice and assessment of statistical predictions. \emph{Journal of the Royal Statistical Society: Series B (Methodological)} 36(2):111--133.
    
    \bibitem[{Sun et~al.(2013)Sun, Wang, \protect\BIBand{} Fang}]{sun2013consistent}
    Sun W, Wang J, Fang Y (2013) Consistent selection of tuning parameters via variable selection stability. \emph{Journal of Machine Learning Research} 14:3419--3440.
    
    \bibitem[{Tibshirani \protect\BIBand{} Tibshirani(2009)}]{tibshirani2009bias}
    Tibshirani RJ, Tibshirani R (2009) A bias correction for the minimum error rate in cross-validation. \emph{The Annals of Applied Statistics} 3(2):822--829.
    
    \bibitem[{Van~Parys et~al.(2021)Van~Parys, Esfahani, \protect\BIBand{} Kuhn}]{van2021data}
    Van~Parys BP, Esfahani PM, Kuhn D (2021) From data to decisions: Distributionally robust optimization is optimal. \emph{Management Science} 67(6):3387--3402.
    
    \bibitem[{Varma \protect\BIBand{} Simon(2006)}]{varma2006bias}
    Varma S, Simon R (2006) Bias in error estimation when using cross-validation for model selection. \emph{BMC Bioinformatics} 7(1):91.
    
    \bibitem[{Vershynin(2012)}]{vershynin2012}
    Vershynin R (2012) Introduction to the non-asymptotic analysis of random matrices. Eldar YC, Kutyniok G, eds., \emph{Compressed Sensing: Theory and Applications}, 210--268 (Cambridge: Cambridge University Press).
    
    \bibitem[{Xu et~al.(2010)Xu, Caramanis, \protect\BIBand{} Mannor}]{xu2010robust}
    Xu H, Caramanis C, Mannor S (2010) Robust regression and {L}asso. \emph{IEEE Transactions on Information Theory} 56(7):3561--3574.
    
    \bibitem[{Ye et~al.(2018)Ye, Yang, \protect\BIBand{} Yang}]{ye2018sparsity}
    Ye C, Yang Y, Yang Y (2018) Sparsity oriented importance learning for high-dimensional linear regression. \emph{Journal of the American Statistical Association} 113(524):1797--1812.
    
    \bibitem[{Yu(2013)}]{yu2013stability}
    Yu B (2013) Stability. \emph{Bernoulli} 19(4):1484--1500.
    
    \end{thebibliography}
    
}
\end{document}